\documentclass[12pt, reqno]{amsart}
\usepackage{amssymb,amsthm,amsmath}
\usepackage[numbers,sort&compress]{natbib}
\usepackage{amssymb,amsmath}
\usepackage{amsfonts}
\usepackage{mathrsfs}
\usepackage{latexsym}
\usepackage{amssymb}
\usepackage{amsthm}
\usepackage{indentfirst}
\usepackage{colortbl}
\date{
Nov 28, 2021}

\let\oldsection\section
\renewcommand\section{\setcounter{equation}{0}\oldsection}

\newtheorem{corollary}{Corollary}[section]
\newtheorem{theorem}{Theorem}[section]
\newtheorem{lemma}{Lemma}[section]
\newtheorem{proposition}{Proposition}[section]
\newtheorem{definition}{Definition}[section]

\newtheorem{remark}{Remark}[section]

\allowdisplaybreaks
\begin{document}

\title[Entropy-bounded solutions compressible Navier-Stokes]{Propagation of uniform boundedness of entropy and inhomogeneous regularities for viscous and heat conductive gases with far field vacuum in three dimensions}



\author{Jinkai~Li}
\address[Jinkai~Li]{South China Research Center for Applied Mathematics and Interdisciplinary Studies,
School of Mathematical Sciences, South China Normal University, Guangzhou 510631, China}
\email{jklimath@m.scnu.edu.cn; jklimath@gmail.com}


\author{Zhouping Xin}
\address[Zhouping Xin]{
The Institute of Mathematical Sciences, The Chinese University of Hong Kong, Hong Kong, China}
\email{zpxin@ims.cuhk.edu.hk}



\begin{abstract}
Due to the highly degeneracy and singularities of the entropy equation, the physical entropy for viscous and heat conductive polytropic gases
behave singularly in the presence of vacuum and it is thus a challenge to study
its dynamics. It is shown in this paper that the uniform boundedness of the entropy and the inhomogeneous Sobolev regularities
of the velocity and temperature can be propagated for viscous and heat conductive gases in $\mathbb R^3$,
provided that the initial vacuum occurs only at far fields
with suitably slow decay of the initial density.
Precisely, it is proved that for any strong solution to the Cauchy problem of the heat conductive compressible
Navier--Stokes equations, the corresponding entropy keeps uniformly bounded and the $L^2$ regularities of the velocity and
temperature can be propagated, up to the existing time of the
solution, as long as the initial density vanishes only at far fields with a rate no more than $O(\frac{1}{|x|^2})$.
The main tools are some singularly weighted energy estimates
and an elaborate De Giorgi type iteration technique. The De Giorgi type iterations are carried out to different equations
in establishing the lower and upper bounds of the entropy.
\end{abstract}

\maketitle

\allowdisplaybreaks
\section{Introduction}

In this paper, we consider the following heat conductive compressible Navier--Stokes equations in $\mathbb R^3$:
\begin{eqnarray}
  \partial_t\rho+\text{div}\,(\rho u)=0,\label{EQRHO}\\
  \rho(\partial_tu+(u\cdot\nabla)u)-\mu\Delta u-(\mu+\lambda)\nabla\text{div}\,u+\nabla p=0, \label{EQU}\\
  c_v\rho(\partial_t\theta+u\cdot\nabla\theta)+p\text{div}\,u-\kappa\Delta\theta=\mathcal Q(\nabla u), \label{EQTHETA}
\end{eqnarray}
where the unknowns $\rho\in[0,\infty)$, $u\in\mathbb R^3$, $\theta\in[0,\infty)$, and $p\in[0,\infty)$, respectively, represent the density, velocity, temperature, and pressure. Here, $c_v$ is a positive constant, $\mu$ and $\lambda$ are
the viscous coefficients, both assumed to be constants and satisfy the physical constraints
$$
\mu>0,\quad 2\mu+3\lambda>0,
$$
$\kappa$ is the heat conductive coefficient, assumed to be a positive constant, and $\mathcal Q(\nabla u)$ is a quadratic term of $\nabla u$ given as
$$
\mathcal Q(\nabla u)=\frac\mu2|\nabla u+(\nabla u)^T|^2+\lambda(\text{div}\,u)^2.
$$

System (\ref{EQRHO})--(\ref{EQTHETA}) is complemented with some constitutive equations. The equations of state for
ideal gases are given by
$$
p=R\rho\theta,\quad e=c_v\theta,
$$
for a positive constant $R$, where $e$ is the specific internal energy. By the Gibbs equation $\theta Ds=De+pD(\frac1\rho)$, where $s$ is
the specific entropy, it holds that
$$
p=Ae^{\frac s{c_v}}\rho^\gamma
$$
for some positive constant $A$, where $\gamma-1=\frac R{c_v}$. It is clear that $\gamma>1$. In terms of $\rho$ and $\theta$, the specific entropy $s$ can be expressed as
\begin{equation}\label{ENTROPY}
s=c_v\left(\log\frac RA+\log\theta-(\gamma-1)\log\rho\right),
\end{equation}
satisfying
\begin{equation}
  \label{EQs}
  \rho(\partial_ts+u\cdot\nabla s)-\frac{\kappa}{c_v}\Delta s=\kappa(\gamma-1)\text{div}\left(\frac{\nabla\rho}{\rho}\right)
  +\frac1\theta\left(\mathcal Q(\nabla u)+\kappa\frac{|\nabla\theta|^2}{\theta}\right),
\end{equation}
in the region where both $\rho$ and $\theta$ are positive.

The compressible Navier--Stokes equations have been studied extensively with many significant
results. One of the important issues in this theory is the vacuum, which, if occurs,
means that the density of the fluid vanishes at some points of the domain occupied by the fluid or at far fields.
Indeed, the possible presence of vacuum is one of the main difficulties in the theory of global well-posedness of general solutions to
the compressible Navier--Stokes equations, and one of the main reasons is that system
(\ref{EQRHO})--(\ref{EQTHETA}) changes its types in the sense that it is a hyperbolic-parabolic
coupled system in the non-vacuum region but degenerates to a hyperbolic-elliptic one near the vacuum region. It is
even more difficult to analyze the properties of the entropy in the presence of vacuum,
as the governing equation (\ref{EQs}) for entropy is highly degenerate and singular in the vacuum region. Due to this,
most of the mathematical studies on the compressible Navier--Stokes equations in the presence
of vacuum focus on system (\ref{EQRHO})--(\ref{EQTHETA}) regardless of the entropy.

The mathematical theory for the one-dimensional compressible Navier--Stokes equations is satisfactory and in particular
the global well-posedness has already been known. In the absence of vacuum, global well-posedness of
strong solutions was established by Kazhikov--Shelukin \cite{KAZHIKOV82} and Kazhikov \cite{KAZHIKOV77}, which
were later extended in the setting of weak solutions, see, e.g., \cite{CHEHOFTRI00,JIAZLO04,ZLOAMO97,ZLOAMO98}; see
Li--Liang \cite{LILIANG16} for the large time behavior of solutions with large initial data. In the presence of vacuum,
the corresponding global well-posedness of strong solutions were recently established by the first author of this paper
in \cite{LJK1DHEAT,LJK1DNONHEAT}, for both heat conductive and non-heat conductive ideal gases without considering the entropy.

One major difference between the one-dimensional and multidimensional cases
for the compressible Navier--Stokes equations is
the possible formation of vacuum.
As shown by Hoff--Smoller \cite{HS}, in the one-dimensional case, no vacuum can be formed later in finite time from non-vacuum initial data,
while such a result remains open for the multidimensional case.

Comparing with the one-dimensional case, mathematical theory for the multidimensional compressible
Navier--Stokes equations is far from complete and some fundamental questions remain challenging,
which include the global well-posedness of smooth solutions and uniqueness of weak solutions.
Global existence of finite energy weak solutions with possible vacuum to the isentropic compressible
Navier--Stokes equations was first proved by Lions
\cite{LIONS98,LIONS93}, later improved by Feireisl--Novotn\'y--Petzeltov\'a \cite{FEIREISL01},
Jiang--Zhang \cite{JIAZHA03}, and more recently Bresch--Jabin \cite{BRESCH18}.
For the full compressible Navier--Stokes equations, global existence of variational weak solutions was proved
by Feireisl \cite{FEIREISL04B} under some assumptions on the
equations of state. For suitably regular initial data, the compressible Navier--Stokes equations admit a unique strong
or classic solution, at least in a short time:
in the absence of vacuum, this was proved by Nash \cite{NASH62} and Serrin \cite{SERRIN59} long time ago, and later
developed in many works, see, e.g., \cite{ITAYA71,VOLHUD72,TANI77,VALLI82,LUKAS84}; in the presence of vacuum,
this was proved in \cite{SALSTR93,CHOKIM04,CHOKIM06-1,CHOKIM06-2} under some initial compatibility conditions
which were removed in \cite{GLLZNOCOM,HUANGNOCOM,LIZHENGNOCOM} recently.
However, global well-posedness of solutions with arbitrary large initial
data is still open.
For the time being, global well-posedness was established only under some additional conditions
on the initial data: the case with small
perturbed initial data around non-vacuum equilibriums was achieved by
Matsumura--Nishida \cite{MATNIS80,MATNIS81,MATNIS82,MATNIS83}, and later developed in many works, see, e.g.,
\cite{PONCE85,VALZAJ86,DECK92,HOFF97,KOBSHI99,DANCHI01,CHENMIAOZHANG10,CHIDAN15,DANXU18,FZZ18}; while the case with
initial data of small energy but
allowing large oscillations and vacuum was proved by Huang--Li--Xin \cite{HLX12} and Li--Xin \cite{LIXIN13} for the
isentropic case, and later generalized to the full system in \cite{HUANGLI11,WENZHU17,LJK3DHEATSMALL}.

It should be noted that significant differences exist in the mathematical theories for the compressible Navier--Stokes
equations between the vacuum and non-vacuum cases and new phenomena may occur depending on the locations and states of vacuum.
In the non-vacuum case, one can establish solutions in both the homogeneous and inhomogeneous spaces depending on the
properties of the initial data, and the
solution spaces guarantee the uniform boundedness of the entropy. This fails in general in the presence of vacuum. Indeed,
in the case that the density has compact support,
the solution, no matter locally or globally, can be established only in the homogeneous spaces, see, e.g., \cite{CHOKIM04,CHOKIM06-1,CHOKIM06-2,HLX12,HUANGLI11,WENZHU17}, but not in the inhomogeneous spaces, see Li--Wang--Xin \cite{LWX}, and the global solutions may have unbounded entropy. Further more, the blowup results of Xin \cite{XIN98}
and Xin--Yan \cite{XINYAN13} imply that the global solutions established
in \cite{HUANGLI11,WENZHU17,LJK3DHEATSMALL} must have unbounded entropy, if initially there is an isolated mass group
surrounded by the vacuum region.
However, it is somewhat surprisingly that if the initial density vanishes only at far fields with a rate no more than
$O(\frac{1}{|x|^2})$, then, as for the non-vacuum case,
the solutions can be established in both the homogeneous and inhomogeneous spaces, and the
entropy can be uniformly bounded up to any finite time, at least in the one-dimensional case, see the recent works
by the authors \cite{LIXINADV,LIXINCPAM}.

Mathematically, since system (\ref{EQRHO})--(\ref{EQTHETA}) is already
closed, one can establish the corresponding theory for it,
regardless of the entropy. However, since the entropy is one of the fundamental states for describing the status
of the fluids, it is physically important to analyze its dynamical behavior. Unfortunately,
the theories developed previously for system (\ref{EQRHO})--(\ref{EQTHETA}) do not provide any
information about the entropy near the vacuum region.

Technically, due to the lack of the expression of the entropy in the vacuum region and the high
singularity and degeneracy of the entropy equation near the vacuum region, in spite of its importance,
the mathematical analysis of the entropy for the viscous compressible fluids in the presence of vacuum was rarely carried out before.
Recently, we have initiated studies on these issues in the one dimensional case in \cite{LIXINADV,LIXINCPAM}.
It was proved in \cite{LIXINADV,LIXINCPAM} that the one-dimensional compressible Navier--Stokes
equations, with or without heat conducting,
can propagate the uniform boundedness of the entropy locally or globally in time, as long as the initial density
vanishes only at far fields with a rate no more than $O(\frac{1}{x^2})$. However, the problems in the multi-dimensional case
have not been studied.

In this paper, we continue our studies, initiated in \cite{LIXINADV,LIXINCPAM},
on the uniform boundedness of the entropy and well-posedness of strong solutions in inhomogeneous spaces
for the multi-dimensional full compressible Navier--Stokes equations
in the presence of vacuum. We will focus on the heat conductive flows. Note that
for the heat conductive case one only need to deal with the
the far field vacuum, as the heat conductivity makes the temperature strictly
positive everywhere after the initial time, which implies that the entropy becomes unbounded instantaneously
if the interior vacuum occurs initially. It is noted that the problem of
the existence of solutions in the inhomogeneous Sobolev spaces, under some
conditions on the initial density allowing vacuum at the far fields has been studied in \cite{LIXINADV} for the non-heat conductive
compressible flows in one dimension, but it has not yet been studied either for the heat conductive flows or in multi dimensions.

We will consider the Cauchy problem only in this paper and, thus, complement the system with the following initial condition:
\begin{equation}
  \label{IC}
  (\rho, u, \theta)|_{t=0}=(\rho_0, u_0, \theta_0).
\end{equation}

The following conventions are used throughout this paper.
For $1\leq q\leq\infty$ and positive integer $m$, $L^q=L^q(\mathbb R^3)$ and
$W^{m,q}=W^{m,q}(\mathbb R^3)$, respectively, are the standard Lebesgue and Sobolev spaces, and $H^m=W^{m,2}$.
$D_0^1=D_0^1(\mathbb R^3)$ and $D^{m,q}=D^{m,q}(\mathbb R^3)$ are the homogeneous Sobolev spaces defined, respectively,
as
\begin{eqnarray*}
D_0^1=\{u\in L^6(\mathbb R^3)|\nabla u\in L^2(\mathbb R^3)\}, \\
D^{m,q}=\{u\in L_{loc}^1(\mathbb R^3)|\nabla^\alpha u\in L^q(\mathbb R^3), 1\leq|\alpha|\leq m\}.
\end{eqnarray*}
For $q=2$, $D^{m,q}$ is simply denoted as $D^m$.
For simplicity, $X$ is used to denote both the space $X$ itself and its $N$ product space $X^N$.
$\|u\|_q$ is the $L^q$ norm of $u$, and $\|(f_1,f_2,\cdots,f_n)\|_X$ is the sum
$\sum_{i=1}^N\|f_i\|_X$ or the equivalent norm $\left(\sum_{i=1}^N\|f_i\|_X^2
\right)^{\frac12}$. The integral of $f$ over $\mathbb R^3$ is abbreviated as
$\int fdx.$

Strong solutions considered in this paper are defined as follows.

\begin{definition}\label{DEF1}
Given a positive time $T$ and assume that
\begin{equation}\label{AID}
0\leq\rho_0\in H^1\cap W^{1,q},\quad u_0\in D_0^1\cap D^2,\quad 0\leq\theta_0\in D_0^1\cap D^2,
\end{equation}
for some $q\in(3,6]$.
A triple $(\rho, u, \theta)$ is called a strong solution to system (\ref{EQRHO})--(\ref{EQTHETA}) in $\mathbb R^3\times(0,T)$,
subject to (\ref{IC}), if it has the regularities
 \begin{eqnarray*}
  \rho\in C([0,T]; H^1\cap W^{1,q}),\quad\rho_t\in C([0,T]; L^2\cap L^q),\\
  (u, \theta)\in C([0,T]; D_0^1\cap D^2)\cap L^2(0,T; D^{2,q}),\\
  (u_t, \theta_t)\in L^2(0,T; D_0^1),\quad \sqrt\rho u_t, \sqrt\rho\theta_t\in L^\infty(0,T; L^2),
\end{eqnarray*}
satisfies equations (\ref{EQRHO})--(\ref{EQTHETA}) a.e.\,$(x,t)\in\mathbb R^3\times(0,T)$, and fulfills the condition
(\ref{IC}).
\end{definition}

\begin{definition}\label{DEF2}
A triple $(\rho, u, \theta)$ is called a global strong solution to system (\ref{EQRHO})--(\ref{EQTHETA}), subject to (\ref{IC}),
if it is a strong solution to the same system in $\mathbb R^3\times(0,T)$, for any finite time $T$.
\end{definition}

Before stating the main result of this paper, let us recall the following two theorems on the
local and global well-posedness of strong solutions to system (\ref{EQRHO})--(\ref{EQTHETA}), subject to (\ref{IC}), which are cited from
\cite{CHOKIM06-2} and \cite{LJK3DHEATSMALL}, respectively.

\begin{theorem}[Local well-posedness, see \cite{CHOKIM06-2}]
  \label{THMLOCAL}
Let $q\in(3,6]$ and assume in addition to (\ref{AID}) that
\begin{equation}
  -\mu\Delta u_0-(\mu+\lambda)\nabla\text{div} u_0+\nabla p_0=\sqrt{\rho_0} g_1, \quad -\kappa\Delta\theta_0-\mathcal Q(\nabla u_0)=\sqrt{\rho_0}g_2,\label{COM}
\end{equation}
for given functions $g_1, g_2\in L^2$, where $p_0=R\rho_0\theta_0$.
Then, there exists a positive time $T_*$ depending only on the initial data, such that system (\ref{EQRHO})--(\ref{EQTHETA}), subject to (\ref{IC}), admits a unique strong solution $(\rho, u, \theta)$ in $\mathbb R^3\times(0,T_*)$.
\end{theorem}

\begin{theorem}[Global well-posedness, see \cite{LJK3DHEATSMALL}]
  \label{THMGLOBAL}
Assume in addition to the conditions (\ref{AID}) and (\ref{COM}) that $2\mu>\lambda$. Then, there is a positive number $\varepsilon_0$ depending only on $R, \gamma, \mu, \lambda,$ and $\kappa$, such that system (\ref{EQRHO})--(\ref{EQTHETA}), subject to (\ref{IC}), has a unique global strong solution, provided that
$$
\mathscr N_0:=\|\rho_0\|_\infty(\|\rho_0\|_3+\|\rho_0\|_\infty^2\|\sqrt{\rho_0}u_0\|_2^2)(\|\nabla u_0\|_2^2+
  \|\rho_0\|_\infty\|\sqrt{\rho_0}E_0\|_2^2)\leq\varepsilon_0.
  $$
\end{theorem}

Note that the solutions in Theorem \ref{THMLOCAL} and Theorem \ref{THMGLOBAL} are both in the homogenous Sobolev spaces
and, as indicated in \cite{LWX}, the inhomogeneous Sobolev regularities can not be propagated by the compressible Navier--Stokes equations
in general, if the initial density is compactly supported. Besides, these solutions may have infinite entropy in the vacuum region,
if the initial density contains interior vacuum. Therefore,
in order to guarantee the uniform boundedness of the entropy or the inhomogeneous Sobolev regularities,
some additional assumptions on the initial density are required, as shown in the following main result of this paper.

\begin{theorem}
\label{THMMAIN}
Assume that in addition to the conditions (\ref{AID}) and (\ref{COM}), the initial density $\rho_0$ is positive on $\mathbb R^3$ and satisfies
\begin{equation}
|\nabla\rho_0(x)|\leq K_1\rho_0^\frac32(x),\quad\forall x\in\mathbb R^3,  \tag{H1}
\end{equation}
for a positive constants $K_1$. Denote
\begin{eqnarray*}
&&s_0=c_v\left(\log\frac RA+\log\theta_0-(\gamma-1)\log\rho_0\right),\quad
\underline s_0=\displaystyle\inf_{x\in\mathbb R^3}s_0(x),\\
&&\overline s_0=\displaystyle\sup_{x\in\mathbb R^3}s_0(x),\quad
\mathscr S_0=-\mu\Delta u_0-(\mu+\lambda)\nabla\text{div}\,u_0+R\nabla(\rho_0\theta_0).
\end{eqnarray*}
Let $(\rho, u, \theta)$ be an arbitrary strong solution to system (\ref{EQRHO})--(\ref{EQTHETA}) in $\mathbb R^3\times(0,T)$, subject to (\ref{IC}), and $s$ the corresponding entropy given by (\ref{ENTROPY}).

Then, the following statements hold:

(i) The additional regularity $u\in L^\infty(0,T; L^2)$ holds, if $u_0\in L^2$.

(ii) The additional regularity $\left(\frac u{\sqrt{\rho}},\theta\right)\in L^\infty(0,T; L^2)$ holds, if $\left(\frac{u_0}{\sqrt{\rho_0}},\theta_0\right)\in L^2$.

(iii) Assume in addition that
\begin{equation}
|\Delta\rho_0(x)| \leq  K_2\rho_0^2(x),\,\quad \forall x\in\mathbb R^3,\tag{H2}
\end{equation}
for a positive constant $K_2$. Then, it holds that
$$
\inf_{ \mathbb R^3\times(0,T)}s(x,t)>-\infty,\quad\text{ as long as }\underline s_0 >-\infty.
$$

(iv) Assume in addition that (H2) holds, then
$$
\sup_{ \mathbb R^3\times(0,T)}s(x,t)<+\infty,
$$
as long as $\overline s_0 <+\infty$ and
$\rho_0^{\frac{1-\gamma}{2}}u_0,\rho_0^{-\frac\gamma2}\nabla u_0,  \rho_0^{1-\frac\gamma2}\theta_0,\rho_0^{1-\frac\gamma2}\nabla\theta_0,
\rho_0^{-\frac\gamma2} \mathscr S_0 \in L^2.$
\end{theorem}

As a corollary of Theorems \ref{THMLOCAL}--\ref{THMMAIN}, we have the following:

\begin{corollary}
  \label{COR}
Assume that the conditions on the initial data in Theorem \ref{THMMAIN} hold.
Then, there is a positive time $T$ depending only on the initial
data but independent of $K_1$ and $K_2$ in (H1)--(H2), such that system
(\ref{EQRHO})--(\ref{EQTHETA}), subject to (\ref{IC}), has a unique strong solution $(\rho, u, \theta)$ in $\mathbb R^3\times(0,T)$, and that
the corresponding entropy given by (\ref{ENTROPY}) is uniformly bounded.
If assume in addition that the conditions in Theorem \ref{THMGLOBAL} hold, then the corresponding entropy of the solution
in Theorem \ref{THMGLOBAL} is uniformly bounded up to any finite time.
\end{corollary}

\begin{remark}[About the conditions (H1)--(H2)]
  (i) Assumptions (H1)--(H2) mean essentially that $\rho_0$ decays at a rate no faster than $O(\frac1{|x|^2})$ at
  far fields. Indeed, for initial density of the form 
  $$
  \rho_0(x)=\frac{K}{\langle x\rangle^{\ell}},\quad K\in(0,\infty), \ell\in[0,\infty), \quad\text{where }\langle x\rangle
  =(1+|x|^2)^\frac12,
  $$
  (H1)--(H2) hold if and only if $\ell\leq2$. For this reason, (H1)--(H2) will be called slow decay assumptions. Such conditions have already been introduced and employed in
  \cite{LIXINADV,LIXINCPAM} to deal with the corresponding problem in one dimension.

  (ii) Set
  $$
  \rho_0=\frac K{\langle x\rangle^\ell},\quad u_0\in C_c^\infty(\mathbb R^3),\quad\theta_0=\frac ARe^{\frac1{c_v}}\rho_0^\gamma,
  $$
with
   $$
  \gamma>\frac54,\quad \max\left\{\frac{1}{2(\gamma-1)},\frac32\right\}<\ell\leq2,\quad K\in(0,\infty).
  $$
  Then, (\ref{AID})--(\ref{COM}) and (H1)--(H2) hold. Therefore, the set of the initial data satisfying conditions in Theorems \ref{THMLOCAL}--\ref{THMMAIN} is not empty.
\end{remark}

\begin{remark}[Propagation of the inhomogeneous regularities]
Theorem \ref{THMMAIN} implies that the inhomogeneous regularity of the velocity can be propagated by
the compressible Navier--Stokes equations in the presence of far field vacuum, as long as the initial density decays sufficiently slowly.
If moreover $\frac{u_0}{\sqrt{\rho_0}}\in L^2$, then the propagation of the inhomogeneous regularity of the temperature holds also. Notably,
this is in sharp contrast to the case with compactly supported initial density \cite{LWX}, where the inhomogeneous regularity
can not be propagated even in a short time.
\end{remark}

\begin{remark}[About the compatibility condition (\ref{COM})]
(i) Assumption (\ref{COM}) is not required and condition (\ref{AID}) can be relaxed in the proof of (iii) of Theorem \ref{THMMAIN}. Besides, some weaker regularities on the
strong solutions than those stated in Definition \ref{DEF1} are sufficient to justify the arguments. In fact,
strong solutions established in \cite{LIZHENGNOCOM} can be chosen to guarantee (iii) of Theorem \ref{THMMAIN}.

(ii) The first condition in (\ref{COM}) is crucially used in the proof of (iv) of Theorem \ref{THMMAIN}, see Proposition \ref{PROP5}, but the second one in (\ref{COM}) is not used. However, since the second condition in (\ref{COM}) is assumed
in the local well-posedness result, Theorem \ref{THMLOCAL}, we still put this assumption in Theorem \ref{THMMAIN}.
It is an interesting question to see
if conclusion (iv) of Theorem \ref{THMMAIN} still holds without assuming (\ref{COM}).
\end{remark}

The main tools of proving Theorem \ref{THMMAIN} are some singularly weighted energy estimates
and an elaborate De Giorgi type iteration technique exploited in our previous works
\cite{LIXINADV,LIXINCPAM}, where the corresponding problems in one dimension were addressed.
Although singular weights chosen in the multidimensional case are exactly the same as those
in the one-dimensional case, in deriving the low order singular energy estimates;
however, in the multi-dimensional case as considered in this paper,
singularly weighted energy estimates for higher order derivatives are also required, which require elaborately chosen
singular weights, see Proposition \ref{PROP4} and Proposition \ref{PROP5}.
Moreover, in the one-dimensional case, the singularly weighted energy estimates for $u_x$ are carried out
in the Lagrangian coordinates via
the dynamical equation of the viscous flux $G$, as this equation can be clearly written out
in the Lagrangian coordinates. However, in higher dimensions, estimates of $\nabla u$ require the control of both the
effective viscous flux $G$ and the vorticity $\nabla\times v$ that are strongly coupled, for which the Lagrangian formulation
has no advantages. Thus, we will work directly with the momentum equations in the Eulerian coordinates and derive the singular energy
estimates for both $\nabla u$ and the material derivative $\dot u$, where we used the idea of Hoff to apply the material derivative
to the momentum to derive the governing system for $\dot u$, see Proposition \ref{PROP3} and Proposition \ref{PROP5}.

The rest of this paper is arranged as follow: in Section \ref{SECSINEST} and Section \ref{SECDEGIORI}, we derive important
a priori singular energy estimates and carry out the De Giorgi type iterations with singular weights, respectively, which are used
to prove Theorem \ref{THMMAIN} in Section \ref{SECPROOF}, and the last section is an appendix on some elementary calculations.

\section{singularly weighted a priori estimates}
\label{SECSINEST}
This section is devoted to carrying out some a priori energy estimates
with singular weights which will be used in the next section
to carry out suitable De Giorgi iterations. Throughout this section, as well as the next one, $(\rho, u, \theta)$ is always assumed
to be a strong solution to system (\ref{EQRHO})--(\ref{EQTHETA}) in $\mathbb R^3\times(0,T)$,
subject to (\ref{IC}), for a given positive time $T$. The initial density is assumed to satisfy (H1).

To simplify the notations, we set
\begin{equation}
  \phi(t):=1+\|\sqrt{\rho_0}u\|_\infty^2(t)+\|\nabla u\|_\infty^2(t),\quad \label{PHIT}
\Phi_T:=\int_0^T\phi dt.
\end{equation}

\subsection{A transport estimate on $\rho$}
\begin{proposition}
  \label{PROP1}
It holds that
$$
e^{-C_*\Phi_T}\rho_0(x)\leq \rho(x,t) \leq e^{C_*\Phi_T}\rho_0(x), \quad\forall(x,t)\in\mathbb R^3\times(0,T),
$$
for a positive constant $C_*$ depending only on $K_1$.
\end{proposition}

\begin{proof}
Define $J=\frac{\rho}{\rho_0}.$ Then, (\ref{EQRHO}) implies that
\begin{equation*}
  \partial_tJ+u\cdot\nabla J+\left(\text{div}\,u+u\cdot\frac{\nabla\rho_0}{\rho_0}\right)J=0.
\end{equation*}
Let $X(x,t)$ be the particle path starting from $x$, that is,
$$
\left\{
\begin{array}{l}
  \partial_tX(x,t)=u(X(x,t),t),\\
  X(x,0)=x.
\end{array}
\right.
$$
Then,
\begin{equation*}
  J(X(x,t),t)=e^{-\int_0^t\left(\text{div}\,u+u\cdot\frac{\nabla\rho_0}{\rho_0}\right)(X(x,\tau),\tau)d\tau}.
\end{equation*}
Thanks to this and using the fact that the mapping $X(\cdot,t): \mathbb R\rightarrow\mathbb R$ is bijective for any $t\in(0,T)$, one can easily obtain
\begin{equation*}
  e^{-\int_0^T\left(\|\text{div}\,u\|_\infty+\left\|u\cdot\frac{\nabla\rho_0}{\rho_0}\right\|_\infty\right)dt}
  \leq J\leq   e^{\int_0^T\left(\|\text{div}\,u\|_\infty+\left\|u\cdot\frac{\nabla\rho_0}{\rho_0}\right\|_\infty\right)dt},\quad\mbox{on }\mathbb R^3\times(0,T),
\end{equation*}
which leads to the conclusion by the Cauchy inequality and assumption (H1).
\end{proof}

\subsection{Singularly weighted $L^\infty(L^2)$ estimates}
\begin{proposition}
\label{PROP2}
It holds for any $\alpha>0$ that
$$
\sup_{0\leq t\leq T}\|(\rho_0^{\frac{1-\alpha}{2}}u,\rho_0^{1-\frac\alpha2}\theta)\|_2^2
+\int_0^T\|(\rho_0^{-\frac\alpha2}\nabla u,\rho_0^{\frac{1-\alpha}{2}}\nabla\theta)\|_2^2dt\leq
C\|(\rho_0^{\frac{1-\alpha}{2}}u_0,\rho_0^{1-\frac\alpha2}\theta_0)\|_2^2
$$
for a positive constant $C$ depending only on $\alpha, c_v, \mu, \lambda, \kappa, K_1,$ and $\Phi_T$.
\end{proposition}

\begin{proof}
  For any fixed $0<\delta<1$ set $\rho_{0\delta}=\rho_0+\delta$. Choose a nonnegative cut off function $\chi\in C_0^\infty(B_2)$ satisfying $\chi\equiv1$ on $B_1$ and set $\chi_R(x)=\phi(\frac xR)$ for $R>0$.

  Multiplying (\ref{EQU}) with $\rho_{0\delta}^{-\alpha}u\chi_R^2$ and integrating over $\mathbb R^3$ yield
  \begin{equation}
    \int [\rho(\partial_tu+(u\cdot\nabla)u)-(\mu\Delta u+(\mu+\lambda)\nabla\text{div}\,u)]\cdot\rho_{0\delta}^{-\alpha}u\chi_R^2dx=
    -\int \nabla p\cdot\rho_{0\delta}^{-\alpha} u\chi_R^2dx.
    \label{P2.1'}
  \end{equation}
  (\ref{EQRHO}) implies that
  \begin{eqnarray}
    \int \rho(\partial_tu+(u\cdot\nabla)u)\cdot\rho_{0\delta}^{-\alpha} u\chi_R^2dx
      = \frac12\int\rho\left(\partial_t|u|^2+u\cdot\nabla|u|^2\right)\rho_{0\delta}^{-\alpha}\chi_R^2dx\nonumber\\
     = \frac12\frac{d}{dt}\int\rho|u|^2\rho_{0\delta}^{-\alpha}\chi_R^2dx-\frac12\int\rho u\cdot\nabla\left(\rho_{0\delta}^{-\alpha}\chi_R^2\right)|u|^2dx.\label{P2.2''}
  \end{eqnarray}
  By (H1), it holds that
  \begin{eqnarray}
    \left|\nabla\left(\rho_{0\delta}^{-\alpha}\chi_R^2\right)\right|&=&\left|-\alpha\rho_{0\delta}^{-(\alpha+1)}
    \nabla\rho_{0\delta}\chi_R^2+2\rho_{0\delta}^{-\alpha}\chi_R\nabla\chi_R\right|\nonumber\\
    &\leq&C\left(\rho_{0\delta}^{-(\alpha+1)}\rho_0^\frac32\chi_R^2+(R\delta^\alpha)^{-1}\chi_R 1_{\mathcal C_R}\right)\nonumber\\
    &\leq&C\left(\rho_{0\delta}^{-\alpha}\sqrt{\rho_0}\chi_R^2+(R\delta^\alpha)^{-1}\chi_R1_{\mathcal C_R}\right),\label{AD1}
  \end{eqnarray}
  where $1_{\mathcal C_R}$ is the characteristic function of the set $\mathcal C_R:=B_{2R}\setminus B_R$.
  This and the Sobolev embedding inequality yields that
  \begin{eqnarray*}
    \left|\int\rho u\cdot\nabla\left(\rho_{0\delta}^{-\alpha}\chi_R^2\right)|u|^2dx
    \right|&\leq&
    C\int\rho|u|^3\left(\rho_{0\delta}^{-\alpha}\sqrt{\rho_0}\chi_R^2+(R\delta^\alpha)^{-1}\chi_R\right)dx\\
    &\leq&C\|\sqrt{\rho_0}u\|_\infty\|\sqrt\rho\rho_{0\delta}^{-\frac\alpha2}u\chi_R\|_2^2+C(R\delta^\alpha)^{-1}\|u\|_\infty
    \|\sqrt\rho u\|_2^2\\
    &\leq&C\phi(t)\|\sqrt\rho\rho_{0\delta}^{-\frac\alpha2}u\chi_R\|_2^2+C(R\delta^\alpha)^{-1}\|u\|_\infty
    \|\sqrt\rho u\|_2^2.
  \end{eqnarray*}
  Substituting this into (\ref{P2.2''}) yields
    \begin{eqnarray}
    &&\int \rho(\partial_tu+(u\cdot\nabla)u)\cdot\rho_{0\delta}^{-\alpha} u\chi_R^2dx\nonumber\\
    &\geq&\frac12\frac{d}{dt}\|\sqrt\rho\rho_{0\delta}^{-\frac\alpha2}u\chi_R\|_2^2 -C\phi(t)
    \|\sqrt\rho\rho_{0\delta}^{-\frac\alpha2}u\chi_R\|_2^2-C(R\delta^\alpha)^{-1}\|u\|_\infty
    \|\sqrt\rho u\|_2^2.\label{P2.2'}
  \end{eqnarray}
It follows from integrating by parts that
\begin{equation}
    -\int \Delta u\cdot\rho_{0\delta}^{-\alpha} u\chi_R^2dx=
    \int|\nabla u|^2\rho_{0\delta}^{-\alpha}\chi_R^2dx+\int\partial_iu\cdot u\partial_i\left(\rho_{0\delta}^{-\alpha}
    \chi_R^2\right)dx.\label{1028-1}
\end{equation}
It follows from (\ref{AD1}), the H\"older, Sobolev, and Cauchy inequalities that
\begin{eqnarray*}
    &&\left|\int\partial_iu\cdot u\partial_i\left(\rho_{0\delta}^{-\alpha}
    \chi_R^2\right)dx\right|\\
    &\leq& C\int|u||\nabla u|\left(\rho_{0\delta}^{-\alpha}\sqrt{\rho_0}\chi_R^2+(R\delta^\alpha)^{-1}\chi_R1_{\mathcal C_R}\right)dx\\
    &\leq&C\|\rho_{0\delta}^{-\frac\alpha2}\nabla u\chi_R\|_2\|\sqrt{\rho_0}\rho_{0\delta}^{-\frac\alpha2}
    u\chi_R\|_2+C(R\delta^\alpha)^{-1}\|\nabla u\|_{L^2(\mathcal C_{R})}\|u\|_6\|\chi_R\|_3\\
    &\leq&\frac14\|\rho_{0\delta}^{-\frac\alpha2}\nabla u\chi_R\|_2^2+C\|\sqrt{\rho_0}\rho_{0\delta}^{-\frac\alpha2}
    u\chi_R\|_2^2+C\delta^{-\alpha}\|\nabla u\|_{L^2(\mathcal C_{R})}\|\nabla u\|_2,
\end{eqnarray*}
where $\|\chi_R\|_3\leq CR$ was used. Substituting this into (\ref{1028-1}) yields
\begin{eqnarray}
   -\int \Delta u\cdot\rho_{0\delta}^{-\alpha} u\chi_R^2dx
  &\geq&\frac34\|\rho_{0\delta}^{-\frac\alpha2}\nabla u\chi_R\|_2^2-C\|\sqrt{\rho_0}\rho_{0\delta}^{-\frac\alpha2}u\chi_R\|_2^2\nonumber\\
  &&-C\delta^{-\alpha}\|\nabla u\|_{L^2(\mathcal C_{R})}\|\nabla u\|_2.\label{P2.3'}
\end{eqnarray}
Similarly, one has
  \begin{eqnarray}\label{P2.4'}
    -\int \nabla\text{div}\,u\cdot\rho_{0\delta}^{-\alpha} u\chi_R^2dx&\geq&\frac12\|\rho_{0\delta}^{-\frac\alpha2}\text{div}\,u\chi_R\|_2^2
     -C\|\sqrt{\rho_0}\rho_{0\delta}^{-\frac\alpha2}u\chi_R\|_2^2\nonumber\\
     &&-C\delta^{-\alpha}\|\nabla u\|_{L^2(\mathcal C_{R})}\|\nabla u\|_2.
  \end{eqnarray}
Thanks to (\ref{AD1}), it follows from integrating by parts and the Cauchy inequality that
  \begin{eqnarray*}
    -\int \nabla p\cdot\rho_{0\delta}^{-\alpha}u\chi_R^2dx&=&\int p\Big[u\cdot\nabla\left(\rho_{0\delta}^{-\alpha}
    \chi_R^2\right)+
    \text{div}\,u\rho_{0\delta}^{-\alpha}\chi_R^2\Big]dx\nonumber\\
    &\leq&C\int\rho\theta\left[|\text{div}u|\rho_{0\delta}^{-\alpha}\chi_R^2+|u|
    \left(\rho_{0\delta}^{-\alpha}\sqrt{\rho_0}\chi_R^2+(R\delta^\alpha)^{-1}\chi_R\right)\right]dx\nonumber\\
    &\leq&\frac{\mu+\lambda}2\|\text{div}u\rho_{0\delta}^{-\frac\alpha2}\chi_R\|_2^2
    +C\left(\|\rho\rho_{0\delta}^{-\frac{\alpha}{2}}\theta\chi_R\|_2^2+\|\sqrt\rho\rho_{0\delta}^{-\frac{\alpha}2}u\chi_R
    \|_2^2\right)\\
    &&+C(R\delta^\alpha)^{-1}(\|\sqrt\rho u\|_2^2+\|\sqrt\rho\theta\|_2^2).
  \end{eqnarray*}
  Substituting (\ref{P2.2'}), (\ref{P2.3'})--(\ref{P2.4'}), and the above inequality into (\ref{P2.1'}), and by Proposition \ref{PROP1}, one gets
  \begin{eqnarray}
  &&\frac{d}{dt}\|\sqrt\rho\rho_{0\delta}^{-\frac\alpha2}u\chi_R\|_2^2
   +1.5\mu\|\rho_{0\delta}^{-\frac\alpha2}\nabla u\chi_R\|_2^2\nonumber\\
  &\leq&  C \phi(t)(\|\rho\rho_{0\delta}^{-\frac{\alpha}{2}}\theta\chi_R\|_2^2
   +\|\sqrt\rho\rho_{0\delta}^{-\frac{\alpha}2}u\chi_R\|_2^2)+C \delta^{-\alpha}\|\nabla u\|_{L^2(\mathcal C_{R})}\|\nabla u\|_2\nonumber\\
   &&+C (R\delta^\alpha)^{-1}(\|\sqrt\rho u\|_2^2+\|\sqrt\rho\theta\|_2^2+\|u\|_\infty
    \|\sqrt\rho u\|_2^2).    \label{P2.6'}
  \end{eqnarray}

  Multiplying (\ref{EQTHETA}) with $\rho_0\rho_{0\delta}^{-\alpha}\theta\chi_R^2$ and integrating over $\mathbb R^3$ yield
  \begin{equation}
    \int [c_v\rho(\partial_t\theta+u\cdot\nabla\theta)-\kappa\Delta\theta]\rho_0\rho_{0\delta}^{-\alpha}\theta\chi_R^2
    dx=\int (\mathcal Q(\nabla u)-p\text{div}\,u)\rho_0\rho_{0\delta}^{-\alpha}\theta\chi_R^2 dx. \label{AT-0}
  \end{equation}
  By (\ref{EQRHO}), one deduces
  \begin{eqnarray}
   \int \rho(\partial_t\theta+u\cdot\nabla\theta)\rho_0\rho_{0\delta}^{-\alpha}\theta\chi_R^2 dx = \frac12\int \rho(\partial_t\theta^2+u\cdot\nabla\theta^2)\rho_0\rho_{0\delta}^{-\alpha}\chi_R^2 dx \nonumber\\
     = \frac12\frac{d}{dt}\|\sqrt{\rho\rho_0}\rho_{0\delta}^{-\frac{\alpha}{2}}\theta\chi_R\|_2^2
    -\frac12\int\rho u\theta^2\nabla(\rho_0\rho_{0\delta}^{-\alpha}\chi_R^2)dx. \label{AT-1}
 \end{eqnarray}
 It follows from direct calculations and (H1) that
 \begin{eqnarray}
   |\nabla(\rho_0\rho_{0\delta}^{-\alpha}\chi_R^2)|&=&|\nabla\rho_0\rho_{0\delta}^{-\alpha}\chi_R^2-\alpha\rho_0\rho_{0\delta}^{-(\alpha+1)}
   \nabla\rho_0\chi_R^2+2\rho_0\rho_{0\delta}^{-\alpha}\chi_R\nabla\chi_R|\nonumber\\
   &\leq&C(\rho_0^\frac32\rho_{0\delta}^{-\alpha}\chi_R^2+\rho_0\rho_{0\delta}^{-\alpha}\chi_R|\nabla\chi_R|).\label{AT-2}
 \end{eqnarray}
 Therefore,
 \begin{eqnarray*}
   &&\left|\int\rho u\theta^2\nabla(\rho_0\rho_{0\delta}^{-\alpha}\chi_R^2)dx\right|\\
   &\leq& C\int\rho|u|\theta^2(\rho_0^\frac32\rho_{0\delta}^{-\alpha}\chi_R^2+\rho_0\rho_{0\delta}^{-\alpha}\chi_R|\nabla\chi_R|)dx\\
   &\leq& C\|\sqrt{\rho_0}u\|_\infty\|\sqrt{\rho\rho_0}\rho_{0\delta}^{-\frac{\alpha}{2}}\theta\chi_R\|_2^2+C(R\delta^\alpha)^{-1}
   \|{\rho_0}u\|_\infty\|\sqrt\rho\theta\|_2^2\\
   &\leq& C\phi(t)\|\sqrt{\rho\rho_0}\rho_{0\delta}^{-\frac{\alpha}{2}}\theta\chi_R\|_2^2+C(R\delta^\alpha)^{-1}
   \phi(t)\|\sqrt\rho\theta\|_2^2.
 \end{eqnarray*}
 Substituting the above inequality into (\ref{AT-1}) gives
 \begin{eqnarray}
   \int \rho(\partial_t\theta+u\cdot\nabla\theta)\rho_0\rho_{0\delta}^{-\alpha}\theta\chi_R^2 dx
   &\geq&\frac12\frac{d}{dt}\|\sqrt{\rho\rho_0}\rho_{0\delta}^{-\frac{\alpha}{2}}\theta\chi_R\|_2^2
   -C\phi(t)\|\sqrt{\rho\rho_0}\rho_{0\delta}^{-\frac{\alpha}{2}}\theta\chi_R\|_2^2\nonumber\\
   &&-C(R\delta^\alpha)^{-1}
   \phi(t)\|\sqrt\rho\theta\|_2^2.\label{AT-3}
 \end{eqnarray}
 Integration by parts yields
 \begin{equation}
 -\int \Delta\theta\rho_0\rho_{0\delta}^{-\alpha}\theta\chi_R^2  dx  = \int|\nabla\theta|^2\rho_0\rho_{0\delta}^{-\alpha}\chi_R^2dx+\int\theta\nabla\theta\cdot
 \nabla(\rho_0\rho_{0\delta}^{-\alpha}\chi_R^2)dx. \label{AT-4}
  \end{equation}
 It follows from (\ref{AT-2}) and the Cauchy inequality that
  \begin{eqnarray*}
  &&\left|\int\theta\nabla\theta\cdot\nabla(\rho_0\rho_{0\delta}^{-\alpha}\chi_R^2)dx\right|\\
  &\leq&C\int\theta|\nabla\theta|(\rho_0^\frac32\rho_{0\delta}^{-\alpha}\chi_R^2
  +\rho_0\rho_{0\delta}^{-\alpha}\chi_R|\nabla\chi_R|)dx\\
  &\leq&\frac12\|\sqrt\rho_0\rho_{0\delta}^{-\frac\alpha2}\nabla\theta\chi_R\|_2^2+C\|\rho_0
  \rho_{0\delta}^{-\frac\alpha2}\theta\chi_R\|_2^2+C\|\sqrt{\rho_0}\rho_{0\delta}^{-\frac\alpha2}\theta\nabla\chi_R\|_2^2\\
  &\leq&\frac12\|\sqrt\rho_0\rho_{0\delta}^{-\frac\alpha2}\nabla\theta\chi_R\|_2^2+C\|\rho_0\rho_{0\delta}^{-\frac\alpha2}
  \theta\chi_R\|_2^2+C(R^2\delta^\alpha)^{-1}\|\sqrt{\rho_0}\theta\|_2^2,
  \end{eqnarray*}
  which and (\ref{AT-4}) lead to
  \begin{equation}
    -\int \Delta\theta\rho_0\rho_{0\delta}^{-\alpha}\theta\chi_R^2  dx
    \geq \frac12\|\sqrt\rho_0\rho_{0\delta}^{-\frac\alpha2}\nabla\theta\chi_R\|_2^2-C\|\rho_0\rho_{0\delta}^{-\frac\alpha2}
  \theta\chi_R\|_2^2-C(R^2\delta^\alpha)^{-1}\|\sqrt{\rho_0}\theta\|_2^2.\label{AT-5}
  \end{equation}
  Note that the Cauchy inequality yields that
  \begin{eqnarray}
    &&\int\mathcal Q(\nabla u) \rho_0\rho_{0\delta}^{-\alpha}\theta \chi_R^2 dx\nonumber\\
    &\leq& C\int|\nabla u|^2\rho_0\rho_{0\delta}^{-\alpha}\theta \chi_R^2dx
    \leq C\|\nabla u\|_\infty\|\rho_{0\delta}^{-\frac\alpha2}\nabla u\chi_R\|_2\|\rho_0\rho_{0\delta}^{-\frac\alpha2}\theta\chi_R\|_2\nonumber\\
    &\leq&\frac\mu2\|\rho_{0\delta}^{-\frac\alpha2}\nabla u\chi_R\|_2^2+ C\phi(t)\|\rho_0\rho_{0\delta}^{-\frac\alpha2}\theta\chi_R\|_2^2 \label{AT-6}
  \end{eqnarray}
  and
  \begin{eqnarray}
    &&-\int p\text{div}\,u \rho_0\rho_{0\delta}^{-\alpha}\theta \chi_R^2dx\leq
    C\int\rho\theta|\text{div}u|\rho_0\rho_{0\delta}^{-\alpha}\theta \chi_R^2dx\nonumber\\
    &\leq& C\|\nabla u\|_\infty \|\sqrt{\rho\rho_0}\rho_{0\delta}^{-\frac{\alpha}2}\theta\chi_R\|_2^2
\leq C\phi(t) \|\sqrt{\rho\rho_0}\rho_{0\delta}^{-\frac{\alpha}2}\theta\chi_R\|_2^2.\label{AT-7}
  \end{eqnarray}
Substituting (\ref{AT-3}) and (\ref{AT-5})--(\ref{AT-7}) into (\ref{AT-0}) and by Proposition \ref{PROP1},
one gets that
  \begin{eqnarray}
    &&c_v\frac{d}{dt}\|\sqrt{\rho\rho_0}\rho_{0\delta}^{-\frac{\alpha}2}\theta\chi_R\|_2^2+\kappa \|\sqrt{\rho_0}\rho_{0\delta}^{-\frac{\alpha}{2}}\nabla\theta\chi_R\|_2^2\nonumber\\
    &\leq&\frac\mu2\|\rho_{0\delta}^{-\frac{\alpha}2}\nabla u\chi_R\|_2^2+C\phi(t) \|\sqrt{\rho\rho_0}\rho_{0\delta}^{-\frac{\alpha}2}\theta\chi_R\|_2^2\nonumber\\
    &&+C[(R\delta^\alpha)^{-1}\phi(t)+(R^2\delta^\alpha)^{-1}]\|\sqrt{\rho_0}\theta\|_2^2.\label{P2.11'}
  \end{eqnarray}

 Summing (\ref{P2.6'}) with (\ref{P2.11'}), it follows from Proposition \ref{PROP1} that
  \begin{eqnarray*}
   &&\mu \|\rho_{0\delta}^{-\frac\alpha2}\nabla u\chi_R\|_2^2+\kappa \|\sqrt{\rho_0}\rho_{0\delta}^{-\frac{\alpha}{2}}\nabla\theta\chi_R\|_2^2\\
   && +\frac{d}{dt}( \|\sqrt\rho\rho_{0\delta}^{-\frac\alpha2}u\chi_R\|_2^2+c_v\|\sqrt{\rho\rho_0}\rho_{0\delta}^{-\frac{\alpha}{2}}
   \theta\chi_R\|_2^2)\\
    &\leq&C\phi(t)( \|\sqrt\rho\rho_{0\delta}^{-\frac\alpha2}u\chi_R\|_2^2+c_v\|\sqrt{\rho\rho_0}\rho_{0\delta}^{-\frac{\alpha}{2}}
   \theta\chi_R\|_2^2)+C\delta^{-\alpha}\|\nabla u\|_{L^2(\mathcal C_{R})}\|\nabla u\|_2\\
    &&+C(R\delta^\alpha)^{-1}(\|\sqrt\rho u\|_2^2 +\|u\|_\infty
    \|\sqrt\rho u\|_2^2)\nonumber\\
    &&+C[(R\delta^\alpha)^{-1}\phi(t)+(R^2\delta^\alpha)^{-1}]\|\sqrt{\rho_0}\theta\|_2^2.
  \end{eqnarray*}
  Applying the the Gr\"onwall inequality to the above and noticing that
  \begin{eqnarray*}
  \int_0^T \|\nabla u\|_{L^2(\mathcal C_{R})}\|\nabla u\|_2 dt\rightarrow0,\quad\mbox{as }R\rightarrow\infty, \\
  (1+\|u\|_\infty) \|\sqrt\rho u\|_2^2+\phi(t)\|\sqrt\rho\theta\|_2^2)\in L^1((0,T)),
  \end{eqnarray*}
  guaranteed by the regularities of $(\rho, u, \theta)$, one gets by taking $R\rightarrow\infty$ and by Proposition \ref{PROP1} that
  \begin{align*}
  ( \|\sqrt{\rho_0}\rho_{0\delta}^{-\frac\alpha2}u \|_2^2+c_v\| \rho_0 \rho_{0\delta}^{-\frac{\alpha}{2}}
   \theta \|_2^2)(t)
   & +\int_0^t(\|\rho_{0\delta}^{-\frac\alpha2}\nabla u \|_2^2+ \|\sqrt{\rho_0}\rho_{0\delta}^{-\frac{\alpha}{2}}\nabla\theta \|_2^2)ds\\
    \leq& \ \ C\|(\rho_0^{\frac{1-\alpha}{2}}u_0,\rho_0^{1-\frac\alpha2}\theta_0)\|_2^2.
  \end{align*}
  Taking $\delta\downarrow0$ to the above inequality and by the monotone convergence theorem,
  the conclusion follows.
\end{proof}

\begin{remark}
\label{RKTESTING}
The basic idea of proving Proposition \ref{PROP2} is to choose $\rho_0^{-\alpha}u$ and
$\rho_0^{1-\alpha}\theta$ as testing functions to (\ref{EQU}) and (\ref{EQTHETA}), respectively, in the fashion
of integrating by parts formally on the whole space.
This argument is justified rigorously by choosing $\rho_{0\delta}^{-\alpha}u\chi_R^2$ and
$\rho_0\rho_{0\delta}^{-\alpha}\theta\chi_R^2$ as testing functions to (\ref{EQU}) and (\ref{EQTHETA}), respectively,
and taking the limits as
$R\rightarrow\infty$ and $\delta\rightarrow0$, successively, as presented in the proof of Proposition \ref{PROP2}.
For simplicity of presentation, the energy estimates in the remaining part of this paper
are carried out in a brief way, that is, we test the relevant equations by some singularly weighted functions directly;
however, as already indicated in the proof of Proposition \ref{PROP2}, due to the assumptions (H1) and (H2),
the arguments can be justified rigorously
by adopting similar cutoff, approximations, and taking the limits as $R\rightarrow\infty$ and $\delta\rightarrow0$.
\end{remark}

\subsection{Singularly weighted  $L^\infty(H^1)$ estimates}
\begin{proposition}
  \label{PROP3}
It holds that
    \begin{eqnarray*}
    \sup_{0\leq t\leq T}\|\rho_0^{-\frac\gamma2}\nabla u\|_2^2+\int_0^T\|\rho_0^{\frac{1-\gamma}{2}}\partial_tu\|_2^2dt
    \leq C\|(\rho_0^{\frac{1-\gamma}{2}}u_0, \rho_0^{1-\frac\gamma2}\theta_0,\rho_0^{-\frac\gamma2}\nabla u_0)\|_2^2
  \end{eqnarray*}
for a positive constant $C$ depending only on $c_v, \mu, \lambda, \kappa, K_1,$ and $\Phi_T$.
\end{proposition}

\begin{proof}
  Multiplying (\ref{EQU}) with $\rho_0^{-\gamma}\partial_tu$ and integrating over $\mathbb R^3$ yield
  \begin{eqnarray}
    -\int(\mu\Delta u+(\mu+\lambda)\nabla\text{div}\,u)\cdot\rho_0^{-\gamma}\partial_tudx+\|\sqrt\rho\rho_0^{-\frac\gamma2}\partial_tu\|_2^2
    \nonumber\\
    =-\int\rho(u\cdot\nabla)u\cdot\rho_0^{-\gamma}\partial_tudx-\int\nabla p\rho_0^{-\gamma}\partial_tudx. \label{P3.1}
  \end{eqnarray}
  It follows from integrating by parts, the Cauchy inequality, and Proposition \ref{PROP1} that
  \begin{eqnarray}
    -\int\Delta u\rho_0^{-\gamma}\partial_tudx&=&\frac12\frac{d}{dt}\|\rho_0^{-\frac\gamma2}\nabla u\|_2^2+\int\nabla u:\partial_tu\otimes\nabla\rho_0^{-\gamma}dx\nonumber\\
    &\geq&\frac12\frac{d}{dt}\|\rho_0^{-\frac\gamma2}\nabla u\|_2^2-\frac{1}{8\mu}\|\sqrt\rho\rho_0^{-\frac\gamma2}\partial_t
    u\|_2^2-C\|\rho_0^{-\frac\gamma2}\nabla u\|_2^2,\label{P3.2}
  \end{eqnarray}
  and, similarly,
  \begin{equation}
    -\int\nabla\text{div}u\rho_0^{-\gamma}\partial_tudx\geq\frac12\frac{d}{dt}\|\rho_0^{-\frac\gamma2}\text{div}\,u\|_2^2
    -\frac{\|\sqrt\rho\rho_0^{-\frac\gamma2}\partial_tu\|_2^2}{8(\mu+\lambda)}-C\|\rho_0^{-\frac\gamma2}\nabla u\|_2^2.
    \label{P3.3}
  \end{equation}
  The Cauchy inequality and Proposition \ref{PROP1} yield
  \begin{equation}
    -\int\rho(u\cdot\nabla)u\cdot\rho_0^{-\gamma}\partial_tudx \leq\frac18\|\sqrt\rho\rho_0^{-\frac\gamma2}\partial_tu\|_2^2
    +C\|\sqrt{\rho_0}u\|_\infty^2\|\rho_0^{-\frac\gamma2}\nabla u\|_2^2. \label{P3.4}
  \end{equation}
  Substituting (\ref{P3.2})--(\ref{P3.4}) into (\ref{P3.1}) leads to
  \begin{eqnarray}
    \frac{d}{dt}(\mu\|\rho_0^{-\frac\gamma2}\nabla u\|_2^2+(\mu+\lambda)\|\rho_0^{-\frac\gamma2}\text{div}u\|_2^2)
    +\frac54\|\sqrt\rho\rho_0^{-\frac\gamma2}\partial_tu\|_2^2\nonumber\\
    \leq C\phi(t)\|\rho_0^{-\frac\gamma2}\nabla u\|_2^2-2\int\rho_0^{-\gamma}\nabla p\cdot\partial_tudx. \label{P3.5}
  \end{eqnarray}

  Let $G$ be the effective viscous flux, i.e.,
  \begin{equation}
    \label{G}
    G:=(2\mu+\lambda)\text{div}u-p.
  \end{equation}
  By the Cauchy inequality and Proposition \ref{PROP1}, it follows
  \begin{eqnarray}
    -\int\rho_0^{-\gamma}\nabla p\cdot\partial_tudx&=&\int p(\nabla\rho_0^{-\gamma}\cdot\partial_tu+\rho_0^{-\gamma}\text{div}\partial_tu)dx\nonumber\\
    &=&\frac{d}{dt}\int\rho_0^{-\gamma}\text{div}updx+\int(p\nabla\rho_0^{-\gamma}\cdot\partial_tu-\rho_0^{-\gamma}
    \text{div}u\partial_tp)dx\nonumber\\
    &\leq&\frac{d}{dt}\int\rho_0^{-\gamma}\text{div}updx-\int\rho_0^{-\gamma}
    \text{div}u\partial_tp dx\nonumber\\
    &&+\eta\|\sqrt\rho\rho_0^{-\frac\gamma2}\partial_tu\|_2^2+C_\eta\|\rho_0^{1-\frac\gamma2}\theta\|_2^2\label{P3.6}
  \end{eqnarray}
  for any $\eta>0$.
  Note that
  \begin{eqnarray}
    -\int\rho_0^{-\gamma}
    \text{div}u\partial_tp dx=-\frac{1}{2(2\mu+\lambda)}\frac{d}{dt}\|\rho_0^{-\frac\gamma2}p\|_2^2-\frac{1}{2\mu+\lambda}
    \int\rho_0^{-\gamma}G\partial_tpdx. \label{P3.7}
  \end{eqnarray}
  It follows from (\ref{EQRHO}), (\ref{EQTHETA}), and the equation of state that
  $$
  -\partial_tp=\text{div}(up-\kappa(\gamma-1)\nabla\theta)+(\gamma-1)(\text{div}up-\mathcal Q(\nabla u)).
  $$
  Thanks to this and using Proposition \ref{PROP1}, one deduces
  \begin{eqnarray}
    -\int\rho_0^{-\gamma}G\partial_tpdx&=&\int\rho_0^{-\gamma}G\text{div}(up-\kappa(\gamma-1)\nabla\theta)dx\nonumber\\
    &&+(\gamma-1)\int\rho_0^{-\gamma}G(\text{div}up-\mathcal Q(\nabla u))dx\nonumber\\
    &=&\int[(\gamma-1)\kappa\nabla\theta-up]\cdot(\rho_0^{-\gamma}\nabla G+\nabla\rho_0^{-\gamma}G)dx\nonumber\\
    &&+(\gamma-1)\int\rho_0^{-\gamma}G(\text{div}up-\mathcal Q(\nabla u))dx\nonumber\\
    &\leq&C\int(|\nabla\theta|+\rho_0\theta|u|)[\rho_0^{-\gamma}|\nabla G|+\rho_0^{\frac12-\gamma}(|\nabla u|+\rho_0\theta)]dx\nonumber\\
    &&+C\int\rho_0^{-\gamma}(|\nabla u|+\rho_0\theta)^2|\nabla u|dx\nonumber\\
    &\leq&\eta\|\rho_0^{-\frac{\gamma+1}{2}}\nabla G\|_2^2+C_\eta\|\rho_0^{\frac{1-\gamma}{2}}\nabla\theta\|_2^2\nonumber\\
    &&+C\phi (t)(\|\rho_0^{1-\frac\gamma2}\theta\|_2^2+\|\rho_0^{-\frac\gamma2}\nabla u\|_2^2)\label{P3.8}
  \end{eqnarray}
  for any positive $\eta$, where $|G|\leq C(|\nabla u|+\rho_0\theta)$ has been used.

It follows from (\ref{P3.5})--(\ref{P3.8}) that
  \begin{eqnarray}
    &&\frac{d}{dt}\Big(\mu\|\rho_0^{-\frac\gamma2}\nabla u\|_2^2+(\mu+\lambda)\|\rho_0^{-\frac\gamma2}\text{div}u\|_2^2
    +\frac{1}{2\mu+\lambda} \|\rho_0^{-\frac\gamma2}p\|_2^2\Big)\nonumber\\
    &&-2\frac{d}{dt}\int\rho_0^{-\gamma}\text{div}u pdx+\frac32\|\sqrt\rho\rho_0^{-\frac\gamma2}\partial_tu\|_2^2\nonumber\\
    &\leq&  \eta\|\rho_0^{-\frac{\gamma+1}{2}}\nabla G\|_2^2+C_\eta\|\rho_0^{\frac{1-\gamma}{2}}\nabla\theta\|_2^2 +C\phi (t)\|(\rho_0^{1-\frac\gamma2}\theta,\rho_0^{-\frac\gamma2}\nabla u)\|_2^2\label{P3.9}
  \end{eqnarray}
  for any positive $\eta$.

It remains to estimate $\|\rho_0^{-\frac{\gamma+1}{2}}\nabla G\|_2^2$.
  Denote $\dot u=\partial_tu+(u\cdot\nabla)u$. Then (\ref{EQU}) yields that
  $$
  \Delta G=\text{div}\,(\rho\dot u).
  $$
  Multiplying the above equation by $\rho_0^{-(\gamma+1)}G$ and integrating over $\mathbb R^3$ yield
  \begin{equation}
    \label{ELL1.1}
    -\int\Delta G\rho_0^{-(\gamma+1)}Gdx=-\int\text{div}(\rho\dot u)\rho_0^{-(\gamma+1)}Gdx.
  \end{equation}
  Integrating by parts and using (H1) and the Cauchy inequality lead to
  \begin{equation}
    \label{ELL1.2}
    -\int\Delta G\rho_0^{-(\gamma+1)}Gdx\geq\frac34\|\rho_0^{-\frac{\gamma+1}{2}}\nabla G\|_2^2-C\|\rho_0^{-\frac\gamma2}
    G\|_2^2.
  \end{equation}
 Similarly, one has by Proposition \ref{PROP1} that
  \begin{eqnarray*}
    &&-\int\text{div}(\rho\dot u)\rho_0^{-(\gamma+1)}Gdx\nonumber\\
    &=&\int\rho\dot u\cdot(\rho_0^{-(\gamma+1)}\nabla G+\nabla\rho_0^{-(\gamma+1)}G)dx\nonumber\\
    &\leq&C(\|\sqrt\rho\rho_0^{-\frac\gamma2}\partial_t u\|_2+\|\sqrt{\rho_0}u\|_\infty\|\rho_0^{-\frac\gamma2}\nabla u\|_2)(\|\rho_0^{-\frac{\gamma+1}{2}}\nabla G\|_2+\|\rho_0^{-\frac\gamma2}G\|_2)\nonumber\\
    &\leq&    C(1+\|\sqrt{\rho_0}u\|_\infty^2)(\|\rho_0^{-\frac\gamma2}\nabla u\|_2^2+\|\rho_0^{1-\frac\gamma2}\theta\|_2^2)
    \nonumber\\
    &&+\frac14\|\rho_0^{-\frac{\gamma+1}{2}}\nabla G\|_2^2+C\|\sqrt\rho\rho_0^{-\frac\gamma2}\partial_tu\|_2^2.
  \end{eqnarray*}
  This, together with (\ref{ELL1.1}) and (\ref{ELL1.2}), leads to
  \begin{equation}
    \|\rho_0^{-\frac{\gamma+1}{2}}\nabla G\|_2^2\leq C\|\sqrt\rho\rho_0^{-\frac\gamma2}\partial_tu\|_2^2+
    C\phi(t)(\|\rho_0^{-\frac\gamma2}\nabla u\|_2^2+\|\rho_0^{1-\frac\gamma2}\theta\|_2^2).\label{ELL1.3}
  \end{equation}

Then, (\ref{P3.9}) and (\ref{ELL1.3}) yield
    \begin{eqnarray*}
    \frac{d}{dt}\left(\mu\|\rho_0^{-\frac\gamma2}\nabla u\|_2^2+(\mu+\lambda)\|\rho_0^{-\frac\gamma2}\text{div}u\|_2^2
    +\frac{1}{2\mu+\lambda} \|\rho_0^{-\frac\gamma2}p\|_2^2\right)
    +\|\sqrt\rho\rho_0^{-\frac\gamma2}\partial_tu\|_2^2\nonumber\\
    \leq \ \ 2 \frac{d}{dt}\int\rho_0^{-\gamma}\text{div}u pdx+C \|\rho_0^{\frac{1-\gamma}{2}}\nabla\theta\|_2^2 +C\phi (t)(\|\rho_0^{1-\frac\gamma2}\theta\|_2^2+\|\rho_0^{-\frac\gamma2}\nabla u\|_2^2).
  \end{eqnarray*}
Note that Proposition \ref{PROP2} together with Proposition \ref{PROP1} yield
$$
\sup_{0\leq t\leq T}\left\|\rho_0^{-\frac\gamma2}p\right\|_2^2\leq C\|(\rho_0^{\frac{1-\gamma}{2}}u_0, \rho_0^{1-\frac\gamma2}\theta_0)\|_2^2.
$$
Thanks to the above two and using the Gr\"onwall inequality, the conclusion follows by applying Propositions \ref{PROP1} and \ref{PROP2}  and the Cauchy inequality.
\end{proof}

\begin{proposition}
  \label{PROP4}
  It holds that
$$
    \sup_{0\leq t\leq T}\|\rho_0^{1-\frac\gamma2}\nabla\theta\|_2^2+\int_0^T\|\rho_0^{\frac{3-\gamma}{2}}\partial_t\theta\|_2^2\leq
    C \|(\rho_0^{\frac{1-\gamma}{2}}u_0, \rho_0^{1-\frac\gamma2}\theta_0,\rho_0^{-\frac\gamma2}\nabla u_0,\rho_0^{1-\frac\gamma2}\nabla\theta_0)\|_2^2
$$
for a positive constant $C$ depending only on $c_v, \mu, \lambda, \kappa, K_1,$ and $\Phi_T$.
\end{proposition}

\begin{proof}
Multiplying (\ref{EQTHETA}) with $\rho_0^{2-\gamma}\partial_t\theta$ and integrating the resultant over $\mathbb R^3$,
one can get from the Cauchy inequality and Proposition \ref{PROP1} that
\begin{eqnarray*}
  &&-\kappa\int\Delta\theta\rho_0^{2-\gamma}\partial_t\theta dx+c_v\|\sqrt\rho\rho_0^{1-\frac\gamma2}\partial_t\theta\|_2^2
  \nonumber\\
  &=&\int(\mathcal Q(\nabla u)-p\text{div}u)\rho_0^{2-\gamma}\partial_t\theta dx-c_v\int\rho u\cdot\nabla\theta
  \rho_0^{2-\gamma}\partial_t\theta dx\nonumber\\
  &\leq&\frac{c_v}{4}\|\sqrt\rho\rho_0^{1-\frac\gamma2}\partial_t\theta\|_2^2+C\phi(t)\|(\rho_0^{\frac{1-\gamma}{2}}\nabla u,\rho_0^{\frac{3-\gamma}{2}}\theta,\rho_0^{1-\frac\gamma2}\nabla\theta)\|_2^2.
\end{eqnarray*}
Similar to (\ref{P3.2}), one has
\begin{eqnarray*}
  -\kappa\int\Delta\theta\rho_0^{2-\gamma}\partial_t\theta dx
   \geq \frac\kappa2\frac{d}{dt}\|\rho_0^{1-\frac\gamma2}
  \nabla\theta\|_2^2-\frac{c_v}{4}\|\sqrt\rho\rho_0^{1-\frac\gamma2}\partial_t\theta\|_2^2
  -C \|\rho_0^{1-\frac\gamma2}\nabla\theta\|_2^2.
\end{eqnarray*}
Therefore,
  \begin{eqnarray*}
    \kappa\frac{d}{dt}\|\rho_0^{1-\frac\gamma2}\nabla\theta\|_2^2+c_v\|\sqrt\rho\rho_0^{1-\frac\gamma2}\partial_t\theta\|_2^2\leq
    C\phi(t)\|(\rho_0^{\frac{1-\gamma}{2}}\nabla u,\rho_0^{\frac{3-\gamma}{2}}\theta,\rho_0^{1-\frac\gamma2}\nabla\theta)\|_2^2,
  \end{eqnarray*}
  from which, by the Gr\"onwall inequality and Propositions \ref{PROP1}, \ref{PROP2}, and \ref{PROP3}, the conclusion follows.
\end{proof}

\subsection{Singularly weighted $L^\infty(L^2)$ estimate for $\dot u$}
\begin{proposition}
\label{PROP5}
Recall $\dot u=\partial_tu+(u\cdot\nabla) u$. Then, it holds that
\begin{eqnarray*}
  &&\sup_{0\leq t\leq T}\|
  \rho_0^{1-\frac{\gamma}{2}}\dot u\|_2^2+\int_0^T\|\rho_0^{\frac{1-\gamma}{2}}\nabla\dot u\|_2^2dt \\
  &\leq& C \|(\rho_0^{\frac{1-\gamma}{2}}u_0, \rho_0^{1-\frac\gamma2}\theta_0,\rho_0^{-\frac\gamma2}\nabla u_0,\rho_0^{1-\frac\gamma2}\nabla\theta_0,\rho_0^{-\frac\gamma2}\mathscr S_0)\|_2^2
\end{eqnarray*}
for a positive constant $C$ depending only on $c_v, \mu, \lambda, \kappa, K_1,$ and $\Phi_T$, where
$$
\mathscr S_0:=\mu\Delta u_0+(\mu+\lambda)\nabla\text{div}u_0-R\nabla(\rho_0\theta_0).
$$
\end{proposition}

\begin{proof}
Taking the operator $\partial_t(~\cdot~)+\text{div}\,(u~\cdot~)$ to (\ref{EQU}), noticing that
\begin{eqnarray}
  \partial_t(\rho\dot u_i)+\text{div}\,(u\rho\dot u_i)&=&\rho (\partial_t\dot u_i+u\cdot\nabla\dot u_i), \nonumber\\
  \partial_t\partial_i\text{div}\,u+\text{div}\,(u\partial_i\text{div}\,u)&=&\partial_i\text{div}\,\dot u+\text{div}\,(u\partial_i\text{div}\,u)
  -\partial_i\text{div}\,((u\cdot\nabla)u),\label{P5.1}\\
  \partial_t\Delta u_i+\text{div}\,(u\Delta u_i)&=&\Delta\dot u_i+\text{div}\,(u\Delta u_i)-\Delta(u\cdot\nabla u_i),
  \label{P5.2}\\
  \partial_t\partial_ip+\text{div}\,(u\partial_ip)&=&\partial_i(p_t+\text{div}\,(up))-\text{div}\,(\partial_iup)\nonumber\\
  &=&R\partial_i(\rho\dot\theta)-\text{div}\,(\partial_iup),\nonumber
\end{eqnarray}
and applying Lemma \ref{LEMMA} in the Appendix to the terms $\text{div}\,(u\partial_i\text{div}\,u)
  -\partial_i\text{div}\,((u\cdot\nabla)u)$ and $\text{div}\,(u\Delta u_i)-\Delta(u\cdot\nabla u_i)$ in (\ref{P5.1}) and (\ref{P5.2}), respectively, one obtains
\begin{eqnarray*}
  &&\rho (\partial_t\dot u+(u\cdot\nabla)\dot u)-\mu\Delta\dot u-(\mu+\lambda)\nabla\text{div}\dot u\\
  &=&\text{div}\,((\nabla u)^Tp)-R\nabla(\rho\dot\theta)+\mu\text{div}\,(\nabla u(\text{div}\,uI-\nabla u-(\nabla u)^T))\\
  &&+(\mu+\lambda)\text{div}\,[((\text{div}\,u)^2-\nabla u:(\nabla u)^T)I-\text{div}\,u(\nabla u)^T].
\end{eqnarray*}
Multiplying the above with $\rho_0^{1-\gamma}\dot u$ and integrating by parts, one gets from the Cauchy inequality, (H1), and Proposition \ref{PROP1} that
\begin{eqnarray}
  &&\int\rho(\partial_t\dot u+(u\cdot\nabla)\dot u)\cdot\rho_0^{1-\gamma}\dot udx-\int[\mu\Delta\dot u+(\mu+\lambda)\nabla\text{div}\dot u]\cdot\rho_0^{1-\gamma}\dot udx \nonumber\\
  &=&\int\text{div}\left[(\nabla u)^Tp-R\rho\dot\theta I+\mu\Big(\nabla u\big(\text{div}\,uI-\nabla u-(\nabla u)^T\big)\Big)
  \right]\cdot\rho_0^{1-\gamma}\dot udx \nonumber\\
  &&+(\mu+\lambda)\int\text{div}\Big(\big((\text{div}\,u)^2-\nabla u:(\nabla u)^T\big)I-\text{div}\,u(\nabla u)^T\Big)\cdot\rho_0^{1-\gamma}\dot udx \nonumber\\
  &\leq&C\int(|\nabla u|^2+|\nabla u|\rho_0\theta+\rho|\dot\theta|)(\rho_0^{1-\gamma}|\nabla\dot u|+|\nabla\rho_0^{1-\gamma}|
  |\dot u|)dx\nonumber\\
  &\leq&\eta\|\rho_0^{\frac{1-\gamma}{2}}\nabla\dot u\|_2^2+C\|\rho_0^{\frac{3-\gamma}{2}}\dot\theta\|_2^2+C\phi(t)\|(\rho_0^{\frac{1-\gamma}{2}}\nabla u, \rho_0^{\frac{3-\gamma}{2}}\theta,\rho_0^{1-\frac\gamma2}\dot u)\|_2^2.\label{P5.3}
\end{eqnarray}
Using (\ref{EQRHO}) and (H1), one deduces
\begin{eqnarray}
  &&\int\rho(\partial_t\dot u+(u\cdot\nabla)\dot u)\cdot\rho_0^{1-\gamma}\dot udx\nonumber\\
   &=&\frac12\frac{d}{dt}\int\rho\rho_0^{1-\gamma}|\dot u|^2dx-\frac12\int\rho u\cdot\nabla\rho_0^{1-\gamma}|\dot u|^2dx\nonumber\\
   &\geq& \frac12\frac{d}{dt}\|\sqrt\rho
  \rho_0^{\frac{1-\gamma}{2}}\dot u\|_2^2-C\|\sqrt{\rho_0}u\|_\infty\|\sqrt\rho\rho_0^{\frac{1-\gamma}{2}}\dot u\|_2^2.
  \label{P2.2}
\end{eqnarray}
Integrating by parts, it follows from (H1) and the Cauchy inequality that
  \begin{eqnarray}
    -\int \Delta \dot u\cdot\rho_0^{-\gamma}\dot udx&=&\|\rho_0^{-\frac\gamma2}\nabla \dot u\|_2^2+\int \nabla\dot u:
    \nabla\rho_0^{-\gamma}\otimes \dot udx\nonumber\\
    &\geq&\frac34\|\rho_0^{-\frac\gamma2}\nabla \dot u\|_2^2-C\|\rho_0^{\frac{1-\gamma}{2}}\dot u\|_2^2
    \label{P2.3}
  \end{eqnarray}
and, similarly,
$$
  -\int\nabla\text{div}\dot u\cdot\rho_0^{1-\gamma}\dot udx \geq \frac34\|\rho_0^{\frac{1-\gamma}{2}}\text{div}\dot u\|_2^2-
  C\|\rho_0^{1-\frac\gamma2}\dot u\|_2^2.
$$
Substituting the above three inequalities into (\ref{P5.3}) and using (H1) yield
\begin{eqnarray*}
  &&\frac{d}{dt}\|\sqrt\rho
  \rho_0^{\frac{1-\gamma}{2}}\dot u\|_2^2+\mu\|\rho_0^{\frac{1-\gamma}{2}}\nabla\dot u\|_2^2 \\
  &\leq& C\|\rho_0^{\frac{3-\gamma}{2}}\partial_t\theta\|_2^2+C\phi(t)\|(\rho_0^{\frac{1-\gamma}{2}}\nabla u, \rho_0^{1-\frac\gamma2}
  \nabla\theta, \rho_0^{\frac{3-\gamma}{2}}\theta,\rho_0^{1-\frac\gamma2}\dot u)\|_2^2,
\end{eqnarray*}
from which, by the Gr\"onwall inequality and applying Propositions \ref{PROP1}--\ref{PROP4}, the conclusion follows.
\end{proof}

\subsection{A singularly weighted elliptic estimate}
\begin{proposition}
  \label{PROPELLIPTIC2}
  It holds that
    \begin{equation*}
    \|\nabla(\rho_0^{-\frac\gamma2}u)\|_6+
    \| \rho_0^{-\frac\gamma2}\nabla u \|_6\leq C \|(\rho_0^{\frac{1-\gamma}{2}}u_0, \rho_0^{1-\frac\gamma2}\theta_0,\rho_0^{-\frac\gamma2}\nabla u_0,\rho_0^{1-\frac\gamma2}\nabla\theta_0,\rho_0^{-\frac\gamma2}\mathscr S_0)\|_2^2
  \end{equation*}
for a positive constant $C$ depending only on $c_v, \mu, \lambda, \kappa, K_1,$ and $\Phi_T$, where $\mathscr S_0$ is defined in Proposition \ref{PROP5}.
\end{proposition}

\begin{proof}
Note that
  \begin{eqnarray*}
    \rho_0^{-\frac\gamma2}\Delta u&=&\Delta(\rho_0^{-\frac\gamma2}u)-\text{div}(u\otimes\nabla\rho_0^{-\frac\gamma2})-\nabla u \cdot\nabla\rho_0^{-\frac\gamma2},\\
    \rho_0^{-\frac\gamma2}\nabla\text{div}u&=&\nabla\text{div}(\rho_0^{-\frac\gamma2}u)-\nabla(u\cdot\nabla\rho_0^{-\frac\gamma2})
    -\text{div}u\nabla\rho_0^{-\frac\gamma2},\\
    \rho_0^{-\frac\gamma2}\nabla p&=&\nabla(\rho_0^{-\frac\gamma2}p)-\nabla\rho_0^{-\frac\gamma2}p.
  \end{eqnarray*}
  Therefore, it follows from (\ref{EQU}) that
  \begin{eqnarray*}
    &&\mu\Delta(\rho_0^{-\frac\gamma2}u)+(\mu+\lambda)\nabla\text{div}(\rho_0^{-\frac\gamma2}u)\\
    &=&\rho_0^{-\frac\gamma2}\rho\dot u+\text{div}\left(\mu u\otimes\nabla\rho_0^{-\frac\gamma2}+(\mu+\lambda)u\cdot\nabla\rho_0^{-\frac\gamma2}I+
    \rho_0^{-\frac\gamma2}pI\right) \\
    &&+(\mu\nabla u+(\mu+\lambda)\text{div}uI-pI)\nabla\rho_0^{-\frac\gamma2}.
  \end{eqnarray*}
It follows from the elliptic estimates, (H1), and Proposition \ref{PROP1} that
  \begin{eqnarray}
    \|\nabla(\rho_0^{-\frac\gamma2}u)\|_6&\leq&C\|\rho_0^{-\frac\gamma2}\rho\dot u+(\mu\nabla u+(\mu+\lambda)\text{div}uI-pI)\nabla\rho_0^{-\frac\gamma2}\|_2\nonumber\\
    &&+C\left\|\mu u\otimes\nabla\rho_0^{-\frac\gamma2}+(\mu+\lambda)u\cdot\nabla\rho_0^{-\frac\gamma2}I+
    \rho_0^{-\frac\gamma2}pI\right\|_6\nonumber\\
    & \leq&C(\|(\rho_0^{1-\frac\gamma2}\dot u,\rho_0^{\frac{1-\gamma}{2}}\nabla u,\rho_0^{\frac{3-\gamma}{2}}\theta)\|_2+
    \|(\rho_0^{\frac{1-\gamma}{2}}u,\rho_0^{1-\frac\gamma2}\theta)\|_6).\label{ELL2.1}
  \end{eqnarray}
  By the Sobolev inequality and (H1), one can get
  \begin{equation*}
    \| \rho_0^{\frac{1-\gamma}{2}}u \|_6\leq C(\|\rho_0^{\frac{1-\gamma}{2}}\nabla u\|_2+\|\nabla\rho_0^{\frac{1-\gamma}{2}}u\|_2)\leq C(\|\rho_0^{\frac{1-\gamma}{2}}\nabla u\|_2+\|\rho_0^{1-\frac{\gamma}{2}}u\|_2)
  \end{equation*}
  and, similarly,
  $$
  \| \rho_0^{1-\frac{\gamma}{2}}\theta \|_6\leq C(\|\rho_0^{1-\frac{\gamma}{2}}\nabla \theta\|_2+\|\rho_0^{\frac{3-\gamma}{2}}\theta\|_2).
  $$
  Substituting the above two inequalities into (\ref{ELL2.1}) yields
  \begin{equation*}
    \|\nabla(\rho_0^{-\frac\gamma2}u)\|_6\leq C\|(\rho_0^{1-\frac\gamma2}\dot u,\rho_0^{\frac{1-\gamma}{2}}\nabla u,\rho_0^{1-\frac{\gamma}{2}}\nabla \theta,\rho_0^{1-\frac{\gamma}{2}} u,\rho_0^{\frac{3-\gamma}{2}}\theta)\|_2,
  \end{equation*}
  and further
  \begin{eqnarray*}
    \|\rho_0^{-\frac\gamma2}\nabla u\|_6&=&\|\nabla(\rho_0^{-\frac\gamma2}u)-\nabla\rho_0^{-\frac\gamma2}u\|_6\leq \|\nabla(\rho_0^{-\frac\gamma2}u)\|_6+C\|\rho_0^{\frac{1-\gamma}{2}}u\|_6 \\
    &\leq& C\|(\rho_0^{1-\frac\gamma2}\dot u,\rho_0^{\frac{1-\gamma}{2}}\nabla u,\rho_0^{1-\frac{\gamma}{2}}\nabla \theta,\rho_0^{1-\frac{\gamma}{2}} u,\rho_0^{\frac{3-\gamma}{2}}\theta)\|_2.
  \end{eqnarray*}
  Combining the above two inequalities and applying Propositions \ref{PROP1}--\ref{PROP5}, the conclusion follows.
\end{proof}

\section{The De Giorgi iterations}
\label{SECDEGIORI}
This section is devoted to carrying out suitable De Giorgi iterations, which are preparations for proving
the lower and upper bounds of the entropy in the next section.
Due to the presence of vacuum at the far fields, which causes the degeneracy of the system,
the De Giorgi iterations performed in this section are of singular type, that is, some singular weights are introduced
in the testing functions chosen in the iterations. Moreover, the iterations are applied to
different equations in establishing the lower and upper bounds of the entropy: in dealing with
the lower bound, a De Giorgi iteration is applied to the entropy equation itself, while in dealing with
the upper bound, it is applied to the temperature equation. The singularly weighted energy estimates established in
the previous section play essential roles in the De Giorgi iteration applied to the temperature equation
to deal with the upper bound of the entropy,
but not in dealing with the lower bound of the entropy.

Set
\begin{equation}\label{MT}
M_T=\frac{\kappa(\gamma-1)}{c_v}e^{C_*\Phi_T}[(1+|\gamma-2|)K_1^2+K_2],
\end{equation}
where $C_*$ is the positive constant stated in Proposition \ref{PROP1} and $\Phi_T$ is given by (\ref{PHIT}).

The following De Giorgi type iteration will be used to get the uniform lower bound of
the entropy in the next section.

\begin{proposition}
  \label{PROP6}
Let $M_T$ be given by (\ref{MT}) and define
$$
\tilde s=\log\theta-(\gamma-1)\log\rho_0+M_Tt,\quad \tilde{\underline s}_0=\frac{\underline s_0}{c_v}-\log\frac RA.
$$
Then, the following statements hold:

(i) For any $\ell\leq\tilde{\underline s}_0$,
$$
\sup_{0\leq t\leq T}\left\|(\tilde s-\ell)_-\right\|_2^2(t)+\int_0^T\left\|
\frac{\nabla(\tilde s-\ell)_-}{\sqrt{\rho_0}}\right\|_2^2dt\leq  C,
$$
for a positive constant $C$ depending only on $c_v, \mu, \lambda, \kappa, K_1, T, \Phi_T,$ and the initial data.

(ii) Set
$$
\mathcal Y_\ell=\mathcal Y_\ell(T):=\sup_{0\leq t\leq T}\|(\tilde s-\ell)_-\|_2^2+\int_0^T\|\nabla(\tilde s-\ell)_-\|_2^2dt,\quad \forall\ell\leq\tilde{\underline s}_0,
$$
where $f_-:=-\min\{f,0\}$.
Then,
$$
\mathcal Y_\ell\leq \frac{C}{(m-\ell)^3}\mathcal Y_m^{\frac32},\quad\forall\ell<m\leq\tilde{\underline s}_0,
$$
for a positive constant $C$ depending only on $c_v, \mu, \lambda, \kappa, K_1,$ and $\Phi_T$.
\end{proposition}

\begin{proof}
It follows from (\ref{EQs}) and definition of $\tilde s$ that
\begin{eqnarray*}
  c_v\rho(\partial_t\tilde s+u\cdot\nabla \tilde s)-\kappa\Delta \tilde s&=&c_vM_T\rho+\kappa(\gamma-1)\Delta\log\rho_0-R\rho u\cdot\nabla\log\rho_0\\
  &&-R\rho\text{div}\,u+\kappa\left|\frac{\nabla\theta}{\theta}\right|^2+\frac{\mathcal Q(\nabla u)}{\theta}.
\end{eqnarray*}
Multiplying the above with $-\frac{(\tilde s-\ell)_-}{\rho_0}$ yields
\begin{align}
  -c_v\int\rho (\partial_t\tilde s+&u\cdot\nabla\tilde s)\rho_0^{-1}(\tilde s-\ell)_-dx+\kappa\int\Delta \tilde s\rho_0^{-1}(\tilde s-\ell)_-dx\nonumber\\
  \leq & -\int[c_vM_T\rho+\kappa(\gamma-1)\Delta\log\rho_0]\rho_0^{-1}(\tilde s-\ell)_-dx\nonumber\\
  &+R\int\rho u\cdot\nabla\log\rho_0\rho_0^{-1} (\tilde s-\ell)_-dx
   +R\int\rho\text{div}\,u\rho_0^{-1}(\tilde s-\ell)_-dx\nonumber\\
  =:&I_1(t)+I_2(t)+I_3(t).  \label{LS.1}
\end{align}
In the similar ways as (\ref{P2.2}) and (\ref{P2.3}), one can get
\begin{eqnarray}
  &&-c_v\int\rho (\partial_t\tilde s+u\cdot\nabla \tilde s)\rho_0^{-1}(\tilde s-\ell)_-dx\nonumber\\
  &=&c_v\int\rho (\partial_t(\tilde s-\ell)_-+u\cdot\nabla (\tilde s-\ell)_-)\rho_0^{-1}(\tilde s-\ell)_-dx\nonumber\\
  &\geq&\frac{c_v}{2}\frac{d}{dt}\left\|\sqrt{\frac{\rho}{\rho_0}}(\tilde s-\ell)_-\right\|_2^2-C\|\sqrt{\rho_0}
   u\|_\infty\|(\tilde s-\ell)_-\|_2^2\label{LS.2}
\end{eqnarray}
and
\begin{eqnarray}
  \int\Delta\tilde s\rho_0^{-1}(\tilde s-\ell)_-dx&=&-\int\nabla \tilde s[\rho_0^{-1}\nabla(\tilde s-\ell)_-+\nabla\rho_0^{-1}(\tilde s-\ell)_-]dx
  \nonumber\\
  &=&\int\nabla(\tilde s-\ell)_-[\rho_0^{-1}\nabla(\tilde s-\ell)_-+\nabla\rho_0^{-1}(\tilde s-\ell)_-]dx\nonumber\\
  &\geq&\frac34\left\|
\frac{\nabla(\tilde s-\ell)_-}{\sqrt{\rho_0}}\right\|_2^2-C\|(\tilde s-\ell)_-\|_2^2. \label{LS.3}
\end{eqnarray}
By assumptions (H1)--(H2), Proposition \ref{PROP1}, and recalling the definition of $M_T$ in (\ref{MT}), one deduces that
\begin{eqnarray*}
  &&c_vM_T\rho+\kappa(\gamma-1)\Delta\log\rho_0\\
  &=&\rho_0\left(c_vM_T\frac{\rho}{\rho_0}+\kappa(\gamma-1)\frac{\Delta\rho_0}{\rho_0^2}-\kappa(\gamma-1)\frac{|\nabla\rho_0|^2}
  {\rho_0^3}\right)\\
  &\geq&\rho_0\Big(c_ve^{-C_*\Phi_T}M_T-\kappa(\gamma-1)K_2-\kappa(\gamma-1)K_1^2\Big)\\
  &=&\varrho_0\kappa(\gamma-1)|\gamma-2|K_1^2>0
\end{eqnarray*}
and, thus,
\begin{eqnarray*}
  I_1 =-\int[c_vM_T\rho+\kappa(\gamma-1)\Delta\log\rho_0]\rho_0^{-1}(\tilde s-\ell)_-dx\leq0.
\end{eqnarray*}
Thanks to this, substituting (\ref{LS.2}) and (\ref{LS.3}) into (\ref{LS.1}) yields
\begin{eqnarray}
&&c_v \frac{d}{dt}\left\|\sqrt{\frac{\rho}{\rho_0}}(\tilde s-\ell)_-\right\|_2^2+1.5\kappa \left\|
\frac{\nabla(\tilde s-\ell)_-}{\sqrt{\rho_0}}\right\|_2^2 \nonumber\\
&\leq &C(1+\|\sqrt{\rho_0} u\|_\infty)\|(\tilde s-\ell)_-\|_2^2+2(I_2+I_3). \label{LS.4}
\end{eqnarray}

(i) By assumption (H1), it follows from Propositions \ref{PROP1}--\ref{PROP3} that
\begin{eqnarray*}
  I_2+I_3\leq C(\|\sqrt\rho u\|_2^2+\|\nabla u\|_2^2+\|(\tilde s-\ell)_-\|_2^2)
  \leq C(1+\|(\tilde s-\ell)_-\|_2^2),
\end{eqnarray*}
which together with (\ref{LS.4}) yields
$$
c_v \frac{d}{dt}\left\|\sqrt{\frac{\rho}{\rho_0}}(\tilde s-\ell)_-\right\|_2^2+1.5\kappa \left\|
\frac{\nabla(\tilde s-\ell)_-}{\sqrt{\rho_0}}\right\|_2^2
 \leq C(1+\|\sqrt{\rho_0} u\|_\infty)\|(\tilde s-\ell)_-\|_2^2+C.
$$
Since $\tilde s|_{t=0}\geq\underline s_0$, it is clear that $(s-\ell)_-|_{t=0}=0$ for any $\ell\leq\underline s_0$.
Thanks to this, applying the Gr\"onwall inequality to the above inequality, and by Proposition \ref{PROP1}, one gets the
first conclusion.

(ii)
It follows from Proposition \ref{PROP1}, assumption (H1), and the Cauchy inequality that
\begin{eqnarray}
  2(I_2+I_3)&\leq&C\int(|\sqrt{\varrho_0}u|+|\nabla u|)(\tilde s-\ell)_-dy\nonumber\\
  &\leq& C(\|\sqrt{\rho_0}u\|_\infty^2+\|\nabla u\|_\infty^2)
  \|(\tilde s-\ell)_-\|_2^2+C\int_{\{\tilde s<\ell\}} 1dx. \label{LS.5}
\end{eqnarray}
Substituting (\ref{LS.5}) into (\ref{LS.4}) leads to
\begin{eqnarray}
   &&c_v \frac{d}{dt}\left\|\sqrt{\frac{\rho}{\rho_0}}(\tilde s-\ell)_-\right\|_2^2+ \kappa \left\|
\frac{\nabla(\tilde s-\ell)_-}{\sqrt{\rho_0}}\right\|_2^2
   \nonumber\\
    &\leq& C(\|\sqrt{\rho_0}u\|_\infty^2+\|\nabla u\|_\infty^2)
  \|(\tilde s-\ell)_-\|_2^2+C\int_{\{\tilde s<\ell\}} 1dx.\label{LS.6}
\end{eqnarray}
One can check easily that, for any $m>\ell$,
$$
1\leq\frac{(\tilde s-m)_-}{m-\ell},\qquad\mbox{ on }\{\tilde s<\ell\}\subseteq\{\tilde s<m\}.
$$
Therefore, it follows from the Gagliardo-Nirenberg inequality that
\begin{eqnarray*}
  \int_{\{\tilde s<\ell\}} 1dx&\leq& \int_{\{\tilde s<\ell\}}\left|\frac{(\tilde s-m)_-}{m-\ell}\right|^3dx\\
  &\leq&\int_{\{\tilde s<m\}}\left|\frac{(\tilde s-m)_-}{m-\ell}\right|^3dx=\frac{\|(\tilde s-m)_-\|_3^3}{(m-\ell)^3}\\
  &\leq&\frac{C}{(m-\ell)^3}\|(\tilde s-m)_-\|_2^{\frac32}\|\nabla(\tilde s-m)_-\|_2^{\frac32},
\end{eqnarray*}
which, substituted into (\ref{LS.6}), gives
\begin{eqnarray*}
   &&c_v \frac{d}{dt}\left\|\sqrt{\frac{\rho}{\rho_0}}(\tilde s-\ell)_-\right\|_2^2+ \kappa \left\|
\frac{\nabla(\tilde s-\ell)_-}{\sqrt{\rho_0}}\right\|_2^2
   \nonumber\\
    &\leq& C(1+\|\sqrt{\rho_0}u\|_\infty^2+\|\nabla u\|_\infty^2)
  \|(\tilde s-\ell)_-\|_2^2  \nonumber\\
  &&+\frac{C}{(m-\ell)^3}\|(\tilde s-m)_-\|_2^{\frac32}\|\nabla(\tilde s-m)_-\|_2^{\frac32}.
\end{eqnarray*}
Recalling that $\tilde s|_{t=0}\geq\tilde{\underline s}_0$, applying the Gr\"onwall inequality to the above, and using Proposition \ref{PROP1},
one obtains the second conclusion.
\end{proof}

To derive a uniform upper bound for the entropy, we need the following De Giorgi type iteration.

\begin{proposition}
\label{PROP7}
Let $M_T$ be given by (\ref{MT}) and define
\begin{eqnarray*}
&&\overline S_0=\frac ARe^{\frac{\overline s_0}{c_v}},\qquad \theta_\ell=\theta-\ell e^{M_Tt}\rho_0^{\gamma-1},\quad\forall\ell\in\mathbb R, \\
&&\mathcal Z_\ell=\mathcal Z_\ell(T)=\sup_{0\leq t\leq T}\|\rho_0^{1-\gamma}(\theta_\ell)_+\|_2^2+\int_0^T\|\rho_0^{\frac12-\gamma}\nabla(\theta_\ell)_+\|_2^2
 dt,\quad\forall\ell\geq\overline S_0,
\end{eqnarray*}
where $f_+:=\max\{f,0\}$.

Then, there is a positive constant $C$ depending only on $c_v$, $\gamma$,
$\mu$, $\lambda$, $\kappa$, $K_1,$ $K_2,$ $T,$ $\Phi_T$, and the initial data, such that
\begin{eqnarray*}
\mathcal Z_\ell&\leq& C(\ell^2+1),\quad\forall\ell\geq\overline S_0,\\
\mathcal Z_\ell&\leq&\frac{C \ell^2}{(\ell-m)^3} \mathcal Z_m^{\frac32},\quad \forall\ell>m\geq\overline S_0.
\end{eqnarray*}

\end{proposition}

\begin{proof}
Direct calculations show that
\begin{eqnarray*}
  c_v\rho(\partial_t\theta_\ell+u\cdot\nabla\theta_\ell)-\kappa\Delta\theta_\ell
  &=&- c_v\ell e^{M_Tt}\rho u\cdot\nabla\rho_0^{\gamma-1}-R\rho(\theta_\ell+\ell e^{M_Tt}\rho_0^{\gamma-1})\text{div}\,u\\
  &&+\ell e^{M_Tt}\left(\kappa\Delta\rho_0^{\gamma-1}-c_vM_T\rho_0^{\gamma-1}\rho\right)+\mathcal Q(\nabla u).
\end{eqnarray*}
Multiplying the above with $\rho_0^{1-2\gamma}(\theta_\ell)_+$ and integrating the resultant over $\mathbb R^3$ yield
\begin{eqnarray}
  &&c_v\int\rho(\partial_t\theta_\ell+u\cdot\nabla\theta_\ell)\rho_0^{1-2\gamma}(\theta_\ell)_+dx-\kappa\int
  \Delta\theta_\ell\rho_0^{1-2\gamma}(\theta_\ell)_+dx\nonumber\\
  &=&-c_v\ell e^{M_Tt}\int\rho u\cdot\nabla\rho_0^{\gamma-1}\rho_0^{1-2\gamma}(\theta_\ell)_+dx-R\int\rho\theta_\ell\text{div}\,u
  \rho_0^{1-2\gamma}(\theta_\ell)_+dx\nonumber\\
  &&-R\ell e^{M_Tt}\int\rho\rho_0^{-\gamma}\text{div}\,u(\theta_\ell)_+dx +\ell e^{M_Tt} \int\left(\kappa\Delta\rho_0^{\gamma-1}-c_vM_T\rho_0^{\gamma-1}\rho\right)\rho_0^{1-2\gamma}(\theta_\ell)_+dx\nonumber\\
  &&+\int\mathcal Q(\nabla u)\rho_0^{1-2\gamma}(\theta_\ell)_+dx
   =: II_1+II_2+II_3+II_4+II_5.\label{US.1}
\end{eqnarray}
In the similar ways as in (\ref{P2.2}) and (\ref{P2.3}), one can get
\begin{eqnarray}
  &&c_v\int\rho(\partial_t\theta_\ell+u\cdot\nabla\theta_\ell)\rho_0^{1-2\gamma}(\theta_\ell)_+dx\nonumber\\
  &\geq& \frac{c_v}{2}\frac{d}{dt}\|\sqrt\rho\rho_0^{\frac12-\gamma}(\theta_\ell)_+\|_2^2-C\|\sqrt{\rho_0}u\|_\infty \|\sqrt\rho\rho_0^{\frac12-\gamma}(\theta_\ell)_+\|_2^2\label{1026-1}
\end{eqnarray}
and
\begin{equation}
  -\int
  \Delta\theta_\ell\rho_0^{1-2\gamma}(\theta_\ell)_+dx \geq\frac34\|\rho_0^{\frac12-\gamma}\nabla(\theta_\ell)_+\|_2^2
  -C\|\rho_0^{1-\gamma}(\theta_\ell)_+\|_2^2.\label{1026-2}
\end{equation}
By (H1)--(H2), (\ref{MT}), and Proposition \ref{PROP1},
one deduces
\begin{eqnarray*}
  \kappa\Delta\rho_0^{\gamma-1}-c_vM_T\rho_0^{\gamma-1}\rho
  &=&\rho_0^\gamma\left(\kappa(\gamma-1)\frac{\Delta\rho_0}{\rho_0^2}+\kappa(\gamma-1)(\gamma-2)\frac{|\nabla\rho_0|^2}
  {\rho_0^3}-c_vM_T\frac{\rho}{\rho_0}\right)\\
  &\leq&\rho_0^\gamma\Big(\kappa(\gamma-1)K_2+\kappa(\gamma-1)|\gamma-2|K_1^2-c_vM_Te^{-C_*\Phi_T}\Big)\\
  &=&-\kappa(\gamma-1)\rho_0^\gamma K_1^2 \leq0
\end{eqnarray*}
and, thus,
\begin{equation}\label{1026-3}
II_4=\ell e^{M_Tt} \int\left(\kappa\Delta\rho_0^{\gamma-1}-c_vM_T\rho_0^{\gamma-1}\rho\right)\rho_0^{1-2\gamma}(\theta_\ell)_+dx\leq0.
\end{equation}
Thanks to (\ref{1026-1})--(\ref{1026-3}), one gets from (\ref{US.1}) and Proposition \ref{PROP1} that
\begin{eqnarray}
  &&c_v\frac{d}{dt}\|\sqrt\rho\rho_0^{\frac12-\gamma}(\theta_\ell)_+\|_2^2+1.5\kappa \|\rho_0^{\frac12-\gamma}\nabla(\theta_\ell)_+\|_2^2
  \nonumber\\
  &\leq& C(1+\|\sqrt{\rho_0}u\|_\infty) \|\sqrt\rho\rho_0^{\frac12-\gamma}(\theta_\ell)_+\|_2^2+2(II_1+II_2+II_3+II_5). \label{US.2}
\end{eqnarray}

(i) By Proposition \ref{PROP1} and using (H1), it follows from the H\"older and Cauchy inequalities that
\begin{eqnarray*}
  &&II_1+II_2+II_3+II_5\\
  &\leq& C\ell\|\sqrt{\rho_0} u\|_2\|\rho_0^{1-\gamma}(\theta_\ell)_+\|_2+C\|\nabla u\|_\infty\|\rho_0^{1-\gamma}(\theta_\ell)_+\|_2^2\\
  &&+C\ell\|\nabla u\|_2\|\rho_0^{1-\gamma}(\theta_\ell)_+\|_2+C\|\rho_0^{-\frac\gamma2}\nabla u\|_4^2\|\rho_0^{1-\gamma}(\theta_\ell)_+\|_2\\
  &\leq& C(\ell^2\|\sqrt{\rho_0} u\|_2^2+\ell^2\|\nabla u\|_2^2+\|\rho_0^{-\frac\gamma2}\nabla u\|_4^4)\\
  &&+C(1+\|\nabla u\|_\infty)\|\sqrt\rho\rho_0^{\frac12-\gamma}(\theta_\ell)_+\|_2^2.
\end{eqnarray*}
Due to Propositions \ref{PROP2}, \ref{PROP3}, and \ref{PROPELLIPTIC2}, one gets by the Young inequality that
\begin{equation*}
 \|\sqrt{\rho_0} u\|_2+\|\nabla u\|_2+\| \rho_0^{-\frac\gamma2}\nabla u\|_4^4\leq C(1+\| \rho_0^{-\frac\gamma2}\nabla u\|_2^2+\| \rho_0^{-\frac\gamma2}\nabla u\|_6^6)\leq C.
\end{equation*}
Therefore,
\begin{eqnarray*}
  &&II_1+II_2+II_3+II_5
 \leq  C(\ell^2 +1)+C(1+\|\nabla u\|_\infty)\|\sqrt\rho\rho_0^{\frac12-\gamma}(\theta_\ell)_+\|_2^2,
\end{eqnarray*}
from which and (\ref{US.2}) one gets
\begin{eqnarray*}
  &&c_v\frac{d}{dt}\|\sqrt\rho\rho_0^{\frac12-\gamma}(\theta_\ell)_+\|_2^2+1.5\kappa \|\rho_0^{\frac12-\gamma}\nabla(\theta_\ell)_+\|_2^2
  \nonumber\\
  &\leq& C(1+\|\sqrt{\rho_0}u\|_\infty+\|\nabla u\|_\infty) \|\sqrt\rho\rho_0^{\frac12-\gamma}(\theta_\ell)_+\|_2^2+C(\ell^2+1).
\end{eqnarray*}
Applying the Gr\"onwall inequality to the above inequality, by Proposition \ref{PROP1}, and noticing that
$(\theta_\ell)_+|_{t=0}=0$, for any $\ell\geq\overline S_0$, the first conclusion follows.

(ii)
Using (H1) and Proposition \ref{PROP1}, one can obtain
\begin{eqnarray*}
  II_1+II_3&\leq&C\ell e^{M_Tt}\int(\rho|u|\rho_0^{\frac12-\gamma}(\theta_\ell)_++|\nabla u|\rho_0^{1-\gamma}(\theta_\ell)_+) dx\\
  &\leq& C\ell(\|\sqrt{\rho_0}u\|_\infty+\|\nabla u\|_\infty)\int\rho_0^{1-\gamma}(\theta_\ell)_+ dx\\
  &\leq& C\ell^2\int_{\{\theta_\ell>0\}}1dx+C(\|\sqrt{\rho_0}u\|_\infty^2
  +\|\nabla u\|_\infty^2)\|\rho_0^{1-\gamma}(\theta_\ell)_+\|_2^2, \\
  II_2&\leq&C\|\nabla u\|_\infty\|\rho_0^{1-\gamma}(\theta_\ell)_+\|_2^2.
\end{eqnarray*}
As for $II_5$, by Proposition \ref{PROPELLIPTIC2}, one can deduce from the H\"older and Sobolev inequalities that
\begin{eqnarray*}
  II_5&=&\int\mathcal Q(\nabla u)\rho_0^{1-2\gamma}(\theta_\ell)_+dx\leq  C\int|\rho_0^{-\frac\gamma2}\nabla u|^2\rho_0^{1-\gamma}(\theta_\ell)_+ dx\\
  &\leq&C\|\rho_0^{-\frac\gamma2}\nabla u\|_6^2\|\rho_0^{1-\gamma}(\theta_\ell)_+\|_6\left(\int_{\{\theta_\ell>0\}}1dx\right)^{\frac12}\\
  &\leq&C\left(\|\rho_0^{1-\gamma}\nabla(\theta_\ell)_+\|_2+\|\rho_0^{\frac32-\gamma}(\theta_\ell)_+\|_2\right)
  \left(\int_{\{\theta_\ell>0\}}1dx\right)^{\frac12}\\
  &\leq&\frac\kappa8\|\rho_0^{\frac12-\gamma}\nabla(\theta_\ell)_+\|_2^2+C\|\rho_0^{1-\gamma}(\theta_\ell)_+\|_2^2
  +C \int_{\{\theta_\ell>0\}}1dx,
\end{eqnarray*}
where the following has been used
\begin{eqnarray*}
  \|\rho_0^{1-\gamma}(\theta_\ell)_+\|_6&\leq&C\|\nabla(\rho_0^{1-\gamma}(\theta_\ell)_+)\|_2\leq \left(\|\rho_0^{1-\gamma}\nabla(\theta_\ell)_+\|_2+\|\nabla \rho_0^{1-\gamma}(\theta_\ell)_+\|_2\right)\\
  &\leq&C\left(\|\rho_0^{1-\gamma}\nabla(\theta_\ell)_+\|_2+\|\rho_0^{\frac32-\gamma}(\theta_\ell)_+\|_2\right).
\end{eqnarray*}
Thanks to the estimates on $II_i, i=1,2,\cdots,5$, it follows from (\ref{US.2}) that
\begin{eqnarray}
  &&c_v\frac{d}{dt}\|\sqrt\rho\rho_0^{\frac12-\gamma}(\theta_\ell)_+\|_2^2 +\kappa \|\rho_0^{\frac12-\gamma}\nabla(\theta_\ell)_+\|_2^2\nonumber\\
  &\leq& C(1+\|\sqrt{\rho_0}u\|_\infty^2+\|\nabla u\|_\infty^2)\|\rho_0^{1-\gamma}(\theta_\ell)_+\|_2^2 +C\ell^2 \int_{\{\theta_\ell>0\}}1dx.\label{US.3}
\end{eqnarray}
Since $M_T>0$, one can check that for any $\ell>m$, it holds that
\begin{equation*}
  1\leq e^{-M_Tt}\rho_0^{1-\gamma}\frac{\theta_m^+}{\ell-m}\leq \rho_0^{1-\gamma}\frac{\theta_m^+}{\ell-m},\qquad\mbox{ on }\{\theta_\ell>0\}\subseteq\{\theta_m>0\}.
\end{equation*}
Thanks to this, it follows from the Gagliardo-Nirenberg inequality and (H1) that
\begin{eqnarray*}
  \int_{\{\theta_\ell>0\}} 1dx&\leq&\int_{\{\theta_\ell>0\}} \left|\frac{\varrho_0^{1-\gamma}\theta_m^+}{\ell-m}\right|^3dx
  \leq\frac{C}{(\ell-m)^3}\|\rho_0^{1-\gamma}\theta_m^+\|_2^{\frac32}\|\nabla(\rho_0^{1-\gamma}\theta_m^+)\|_2^{\frac32}\\
  &\leq&\frac{C}{(\ell-m)^3}\|\rho_0^{1-\gamma}\theta_m^+\|_2^{\frac32}(\|\rho_0^{1-\gamma}\nabla\theta_m^+\|_2
  +\|\nabla\rho_0^{1-\gamma}\theta_m^+\|_2)^{\frac32}\\
  &\leq&\frac{C}{(\ell-m)^3}\|\rho_0^{1-\gamma}\theta_m^+\|_2^{\frac32}(\|\rho_0^{1-\gamma}\nabla\theta_m^+\|_2
  +\|\rho_0^{1-\gamma}\theta_m^+\|_2)^{\frac32}.
\end{eqnarray*}
Substituting this into (\ref{US.3}) gives
\begin{eqnarray*}
  &&c_v\frac{d}{dt}\|\sqrt\rho\rho_0^{\frac12-\gamma}(\theta_\ell)_+\|_2^2 +\kappa \|\rho_0^{\frac12-\gamma}\nabla(\theta_\ell)_+\|_2^2\nonumber\\
  &\leq& \frac{C\ell^2}{(\ell-m)^3} \|\rho_0^{1-\gamma}\theta_m^+\|_2^{\frac32}(\|\rho_0^{1-\gamma}\nabla\theta_m^+\|_2
  +\|\rho_0^{1-\gamma}\theta_m^+\|_2)^{\frac32} \nonumber\\
  &&+C(1+\|\sqrt{\rho_0}u\|_\infty^2+\|\nabla u\|_\infty^2)\|\rho_0^{1-\gamma}(\theta_\ell)_+\|_2^2.
\end{eqnarray*}
Note that $(\theta_\ell)_+|_{t=0}=0$, for any $\ell\geq\overline S_0$.
The second conclusion follows from the above inequality by the Gr\"onwall inequality and Proposition \ref{PROP1}.
\end{proof}

\section{Uniform boundedness of entropy: proof of Theorem \ref{THMMAIN}}
\label{SECPROOF}
Let us first recall the following lemma cited from \cite{LIXINCPAM}.

\begin{lemma}[\cite{LIXINCPAM}]
  \label{lemiteration}
Let $m_0\in[0,\infty)$ be given and $f$ be a nonnegative non-increasing function on $[m_0,\infty)$ satisfying
$$
f(\ell)\leq\frac{M_0(\ell+1)^\alpha}{(\ell-m)^\beta}f^\sigma(m),\quad\forall\ell>m\geq m_0,
$$
for some nonnegative constants $M_0, \alpha, \beta,$ and $\sigma$, with $0\leq\alpha<\beta$ and $\sigma>1$.
Then,
$$f(m_0+d)=0,$$
where
$$d=\left[2 f^\sigma(m_0) (m_0+M_0+2)^{\frac{2\alpha+2\beta+1}{\sigma-1}+\frac{\beta}{(\sigma-1)^2}+2\alpha+\beta+1}\right]^{\frac{1}{\beta-\alpha}}+2.
$$
\end{lemma}

Then, we can prove Theorem \ref{THMMAIN} as follows.

\begin{proof}[Proof of Theorem \ref{THMMAIN}]
(i) and (ii) are direct corollaries of Proposition \ref{PROP2}, by choosing $\alpha=1$ and $\alpha=2$, respectively.

(iii) Let $\tilde s, \tilde{\underline s}_0,$ and $\mathcal Y_\ell$ be defined as in Proposition \ref{PROP6}. By Proposition \ref{PROP6}, $\mathcal Y_\ell$ is finitely valued for $\ell\leq\tilde{\underline s}_0$ and
  \begin{equation}
  \label{INEQ1}
  \mathcal Y_\ell\leq\frac{C_1}{(m-\ell)^3}\mathcal Y_m^\frac32,\quad\forall\ell<m\leq\tilde{\underline s}_0,
  \end{equation}
  for a positive constant $C_1$ depending only on $c_v, \mu, \lambda, \kappa, K_1,$ and $\Phi_T$. Define $f(\ell):=\mathcal Y_{\tilde{\underline s}_0-\ell}$. Noticing that
  $\mathcal Y_\ell$ is nonnegative and nondecreasing with respect to $\ell$, it is clear that $f(\ell)$ is nonnegative and non-increasing with respect to $\ell\in[0,\infty)$. Then, it follows from (\ref{INEQ1}) that
  $$
 f(\ell)\leq\frac{C_1}{(m-\ell)^3} f^\frac32(m),\quad\forall \ell>m\geq0,
  $$
  from which, by Lemma \ref{lemiteration}, it follows that $f(d_1)=\mathcal Y_{\tilde{\underline s}_0-d_1}=0$, for a positive constant $d_1$ depending only on
  $C_1$ and $f(0)=\mathcal Y_{\tilde{\underline s}_0}$. With the aid of this and recalling the definition of $\mathcal Y_\ell$, one gets
  $(\tilde s-\tilde{\underline s}_0+d_1)_-=0$ a.e.\,in $\mathbb R^3\times(0,T)$, and, consequently,
  $$
  \inf_{(x,t)\in\mathbb R^3\times(0,T)}\tilde s(x,t)\geq\tilde{\underline s}_0-d_1.
  $$
  Therefore, recalling the definition of $\tilde s$ and using Proposition \ref{PROP1}, one deduces
  \begin{eqnarray*}
  s(x,t)&\geq& c_v\inf_{(x,t)\in\mathbb R^3\times(0,T)}\tilde s+c_v\log\frac RA+(\gamma-1)\log\left(
    \inf_{(x,t)\in\mathbb R^3\times(0,T)}\frac{\rho_0}{\rho}\right)\\
    &\geq& c_v\left(\tilde{\underline s}_0-d_1+\log\frac RA\right)-(\gamma-1)C\Phi_T,\quad\forall(x,t)\in\mathbb R^3\times(0,T),
  \end{eqnarray*}
  for a positive constant $C$ depending only on $K_1.$ This proves the first conclusion.

  (iv) Let $\vartheta_\ell, \overline S_0,$ and $\mathcal Z_\ell$ be defined as in Proposition \ref{PROP7}. It follows from Proposition \ref{PROP7} that $\mathcal Z_\ell$ is finitely valued for any $\ell\geq\overline S_0$ and
  \begin{equation*}
    \mathcal Z_\ell\leq\frac{C_2\ell^2}{(\ell-m)^3}\mathcal Z_m^\frac32,\quad\forall\ell>m\geq\overline S_0,
  \end{equation*}
  for a positive constant $C_2$ depending only on $c_v, \gamma, \mu, \lambda, \kappa, K_1, T, \Phi_T,$ and the initial data. Then, by Lemma \ref{lemiteration},
  there is a positive constant $d_2$ depending only on $C_2$ and $\mathcal Z_{\overline S_0}$, such that $\mathcal Z_{\overline S_0+d_2}=0$.
  Thanks to this and recalling the definition of $\mathcal Z_\ell$, one gets
  $$
  \theta(x,t)\leq(\overline S_0+d_2)\rho_0^{\gamma-1}(x),\quad\forall (x,t)\in\mathbb R^3\times(0,T),
  $$
  from which, since $\theta=\frac ARe^{\frac{s}{c_v}}\rho^{\gamma-1}$ and by Proposition \ref{PROP1}, it follows
  $$
  s\leq c_v\log\left[\frac RA(\overline S_0+d_2)\right]+R\log\frac{\rho_0}{\rho}\leq c_v\log\left[\frac RA(\overline S_0+d_2)\right]+RC\Phi_T,
  $$
  for a positive constant $C$ depending only on $K_1$, which yields the second conclusion. Thus, Theorem \ref{THMMAIN} is proved.
\end{proof}

\section{Appendix}

In this appendix, we prove the following lemma which has already been used in the proof of Proposition \ref{PROP5}.

\begin{lemma}
\label{LEMMA}
The following two identities hold
\begin{eqnarray*}
&\text{div}\,(\Delta u\otimes u)-\Delta((u\cdot\nabla)u)=\text{div}\,(\nabla u(\text{div}\,uI-\nabla u-(\nabla u)^t),\\
&\text{div}\,(\nabla\text{div}\,u\otimes u)-\nabla\text{div}\,((u\cdot\nabla)u)
  =\text{div}\,[((\text{div}\,u)^2-\nabla u:(\nabla u)^t)I-\text{div}\,u(\nabla u)^t].
\end{eqnarray*}
\end{lemma}

\begin{proof}
By direct calculations,
\begin{eqnarray*}
  &&\text{div}\,(u\Delta u_i)-\Delta(u\cdot\nabla u_i)\\
  &=&\text{div}(u\partial_k^2u_i)-\partial_k^2(u\cdot\nabla u_i)\\
  &=&\text{div}\,(\partial_k(u\partial_ku_i)-\partial_ku\partial_ku_i)-\partial_k(u\cdot\nabla\partial_ku_i)-\partial_k(\partial_ku\cdot\nabla
  u_i)\\
  &=&\partial_k(u\cdot\nabla\partial_ku_i+\text{div}\,u\partial_ku_i)-\text{div}\,(\partial_ku\partial_ku_i)\\
  &&-\partial_k(u\cdot\nabla\partial_ku_i)-\partial_k(\partial_ku\cdot\nabla u_i)\\
  &=&\partial_k(\text{div}\,u\partial_ku_i-\partial_ku\cdot\nabla u_i)-\text{div}\,(\partial_ku\partial_ku_i),
\end{eqnarray*}
the first conclusion follows.
One deduces
\begin{eqnarray*}
  &&\text{div}\,(u\partial_i\text{div}\,u)-\partial_i\text{div}\,((u\cdot\nabla)u)\\
  &=&\text{div}\,\partial_i(u\text{div}\,u)-\text{div}\,(\partial_iu\text{div}\,u)-\partial_i(u\cdot\nabla\text{div}\,u+\nabla u_k\cdot\partial_k u)\\
  &=&\partial_i(u\cdot\nabla\text{div}\,u+(\text{div}\,u)^2)-\text{div}\,(\partial_iu\text{div}\,u)\\
  &&-\partial_i(u\cdot\nabla\text{div}\,u+\nabla u_k\cdot\partial_ku)\\
  &=&\partial_i((\text{div}\,u)^2-\nabla u_k\cdot\partial_ku)-\text{div}\,(\partial_iu\text{div}\,u),
\end{eqnarray*}
the second conclusion follows.
\end{proof}

\section*{Acknowledgments}
{The work of J. L. was supported in part by the National
Natural Science Foundation of China (11971009 and 11871005), by
the Guangdong Basic and Applied Basic Research Foundation (2019A1515011621,
2020B1515310005, and 2021A1515010247), and by the Key Project of National Natural Science Foundation of China (12131010).
The work of Z.X. was supported in part by the Zheng Ge Ru Foundation, by the Hong Kong RGC Earmarked Research Grants CUHK-14305315,
CUHK-14300917, CUHK-14302819, and CUHK-14302917, and by the Guangdong Basic and Applied Basic Research
Foundation 2020B1515310002.}
\par

\end{document}